\documentclass[reqno]{amsart}
\usepackage[english]{babel}
\usepackage{amssymb,enumerate,url}
\usepackage{palatino,mathpazo}
\usepackage{needspace}
\usepackage[pagebackref,colorlinks]{hyperref}
\newcommand{\Beta}{\mathrm{B}}
\newcommand{\llangle}{{\langle\!\langle}}
\newcommand{\rrangle}{{\rangle\!\rangle}}
\newenvironment{enumeratenumeric}{\begin{enumerate}[1.] }{\end{enumerate}}
\numberwithin{equation}{section}
\newtheorem{theorem}{Theorem}[section]
\newtheorem{corollary}[theorem]{Corollary}
\newtheorem{definition}[theorem]{Definition}
\newtheorem{lemma}[theorem]{Lemma}
\newtheorem{proposition}[theorem]{Proposition}
\theoremstyle{remark}
\newtheorem{remark}[theorem]{Remark}
\theoremstyle{plain}

\begin{document}
\begin{sloppypar}

\title[A Converse to the Skoda $L^2$ division theorem II]
{A Converse to the Skoda $L^2$ division theorem II}

\author[Z. Li]{Zhi Li}
\address{Zhi Li: School of Mathematical Sciences and Key Laboratory of Mathematics and Information Networks (Ministry of Education), Beijing University of Posts and Telecommunications, Beijing 100876, China}
\email{lizhi@amss.ac.cn, lizhi10@foxmail.com}

\author[X. Meng]{Xiankui Meng}
\address{Xiankui Meng: School of Mathematical Sciences and Key Laboratory of Mathematics and Information Networks (Ministry of Education), Beijing University of Posts and Telecommunications, Beijing 100876, China}
\email{mengxiankui@amss.ac.cn}

\author[J. Ning]{Jiafu Ning}
\address{Jiafu Ning: School of Mathematics and Statistics, HNP-LAMA, Central South University, Changsha, Hunan 410083, China}
\email{jfning@csu.edu.cn}

\author[X. Zhou]{Xiangyu Zhou}
\address{Xiangyu Zhou: Institute of Mathematics, Academy of Mathematics and Systems Sciences, and Hua Loo-Keng Key
	Laboratory of Mathematics, Chinese Academy of
	Sciences, Beijing 100190, China}
\email{xyzhou@math.ac.cn}


\thanks{
This work is supported partially by the National Key R\&D Program of China (No.2021YFA1002600 and No.2021YFA1003100).  Z. Li and X. Meng are supported partially by the National Natural Science Foundation of China (No.12271057). J. Ning is supported partially by NSFC Grant No. 12671103. X. Zhou is supported partially by NSFC Grant No. 12288201.
}

\begin{abstract}
We prove a converse to the original Skoda $L^2$ division theorem, asserting that Skoda $L^2$ division estimates for a single nonzero holomorphic function force the weight to be plurisubharmonic. In addition, we establish corresponding results for higher-rank quotient bundles. Finally, we derive an optimal $L^2$ extension theorem directly from the division estimates.
\end{abstract}

\maketitle

\section{Introduction}

Let $D\subset\mathbb C^n$ be a domain.  Given a holomorphic function
$f$ and a holomorphic tuple
$g=(g_1,\ldots,g_m)$ on $D$, the division problem is to determine whether there exists a holomorphic tuple
$h=(h_1,\ldots,h_m)$ on $D$ such that
\[
 f=g\cdot h:=\sum_{j=1}^m g_jh_j.
\]
Oka's theorem \cite{oka1950} solves the qualitative division problem on pseudoconvex domains when $f$ is locally in the ideal generated by the
components of $g$.

Skoda's theorem gives an $L^2$
version of the division theorem, replacing the local ideal membership by a weighted $L^2$ condition, which leads
to a global holomorphic division with an explicit $L^2$ estimate.
Skoda's $L^2$ theorem provides both solvability and $L^2$ estimates for the
division problem. In order to state Skoda's theorem, we now fix some notations used throughout the paper.  Write
\[
 |g|^2=\sum_{j=1}^m|g_j|^2,\qquad
 |h|^2=\sum_{j=1}^m|h_j|^2,\qquad
 p=\min\{n,m-1\}.
\]
For $\varepsilon>0$, set
\begin{equation}\label{equ:skoda-weight}
 \psi_{g,\varepsilon}:=p(1+\varepsilon)\log|g|^2,
\end{equation}
and
\begin{equation}\label{equ:skoda-datum}
 J_{\varphi,g,\varepsilon}(f)
 :=\int_D\frac{|f|^2}{|g|^2}
 e^{-\varphi-\psi_{g,\varepsilon}} dV.
\end{equation}

\begin{theorem}[Skoda \cite{skoda1972,skoda1978}]
\label{thm:skoda-l2-division}
Let $D\subset\mathbb C^n$ be a pseudoconvex domain, $\varphi$ a
plurisubharmonic function on $D$ and $g=(g_1,\ldots,g_m)$
a holomorphic tuple in $\mathcal O(D)^{\oplus m}$ whose components
have no common zero.  Set
$p=\min\{n,m-1\}$.  If $\varepsilon>0$ and
$J_{\varphi,g,\varepsilon}(f)<\infty$, then there exists
$h=(h_1,\ldots,h_m)\in\mathcal O(D)^{\oplus m}$ such that $f=g\cdot h$ and
\begin{equation}\label{equ:classical-skoda}
 \int_D|h|^2e^{-\varphi-\psi_{g,\varepsilon}} dV
 \leqslant\left(1+\frac1\varepsilon\right)
 J_{\varphi,g,\varepsilon}(f).
\end{equation}
\end{theorem}

Theorem \ref{thm:skoda-l2-division} gives an analytic criterion for ideal membership.
It underlies the analytic form of the Brian\c{c}on--Skoda
phenomenon \cite{briancon-skoda1974} as well as effective forms of the
 Nullstellensatz \cite{brownawell1987,ein-lazarsfeld1999}.
With Skoda's result, Siu established the deformation invariance of plurigenera and the finite
generation of the canonical ring of complex projective manifolds of
general type (see {\cite{siu1998,siu2002,siu2004,siu2005,siu2009tech}}). Skoda-type division theorems have
also been generalized to the Koszul complex and
exact sequences \cite{ji2012,ji2013}.\\

Zhou and his coauthors systematically initiated the study of the converse problems in $L^2$ theory, which have recently attracted considerable
attention in several complex variables.  The guiding question is whether an
analytic result with $L^2$ estimates known to follow from plurisubharmonicity or curvature
positivity can, conversely, recover that positivity.  Optimal and multiple
coarse $L^2$ extension properties characterize plurisubharmonic functions,
while $L^2$ existence for $\bar\partial$ with optimal $L^2$ estimates or its vector-bundle
analogue characterizes Griffiths or Nakano positivity
\cite{deng-ning-wang2021,deng-ning-wang-zhou2023,deng-wang-zhang-zhou2024new-char}.  

In
\cite{li-meng-ning-zhou2025arxiv}, we obtained a converse to the strong
$L^2$ division theorem using a formally stronger $L^2$ estimate in place of Skoda's original estimate.  The present paper, which is a sequel to \cite{li-meng-ning-zhou2025arxiv}, shows that the original Skoda
estimate also detects plurisubharmonicity.\\

Our first main result is the converse to the classical
Skoda $L^2$ division theorem, i.e., Theorem \ref{thm:skoda-l2-division}. To state our theorem more concisely, we make the following definition.

\begin{definition}
    Let $\varphi\in L^1_{\mathrm{loc}}(D)$ and $f\in\mathcal O(D)$. We say that $f$ is \emph{Skoda divisible} with
 respect to $\varphi$ if, for every $m\geqslant1$, every holomorphic tuple
 $g\in\mathcal O(D)^{\oplus m}$ whose components have no common zero, and
 every $\varepsilon>0$ satisfying
 $J_{\varphi,g,\varepsilon}(f)<\infty$, there is a holomorphic tuple
 $h\in\mathcal O(D)^{\oplus m}$ such that $g\cdot h=f$ and
 \[
 \int_D|h|^2e^{-\varphi-\psi_{g,\varepsilon}} dV
 \leqslant\left(1+\frac1\varepsilon\right)
 J_{\varphi,g,\varepsilon}(f).
 \]
\end{definition}
Denote by $A^2(D,\varphi)$ the Bergman space with respect to $\varphi$. With these notations, the converse of Theorem \ref{thm:skoda-l2-division} can be stated as follows.

\begin{theorem}
\label{thm:main-inverse-general}
Let $D\subset\mathbb C^n$ be a bounded domain and $\varphi\in C^2(D)$.
Assume that there is a non-zero $f\in A^2(D,\varphi)$ that is Skoda
divisible with respect to $\varphi$.  Then $\varphi$ is
plurisubharmonic on $D$.
\end{theorem}

Theorem \ref{thm:main-inverse-general} gives the equivalence between the plurisubharmonicity and Skoda divisibility in pseudoconvex domains. In our previous work \cite{li-meng-ning-zhou2025arxiv}, we gave an equivalence between the plurisubharmonicity and (strongly) $L^2$ divisibility.  In \cite{li-meng-ning-zhou2025arxiv}, an $f\in\mathcal{O}(D)$ is said to be strongly $L^2$ divisible with respect to $\varphi$ if
for every $m\geqslant1$, every holomorphic tuple
 $g\in\mathcal O(D)^{\oplus m}$ whose components have no common zero, and
 every smooth plurisubharmonic function $\psi$ for which
 $B_{\psi,g}=\left[i\partial\bar\partial\psi\otimes
 \operatorname{Id}_{\underline{\mathbb C}^{m}}
 +i\beta^*\wedge\beta,\Lambda\right]>0$, the finiteness of
 \[
 I_{\varphi,\psi,g}(f)
 =\int_D\left\{\frac1{|g|^2}
 +\left\langle B_{\psi,g}^{-1}\beta^*\wedge dz,
 \beta^*\wedge dz\right\rangle\right\}
 |f|^2e^{-\varphi-\psi} dV
 \]
 implies that
 there is a holomorphic tuple $h\in\mathcal O(D)^{\oplus m}$ such that
 $g\cdot h=f$ and
 \[
  \int_D|h|^2e^{-\varphi-\psi} dV
  \leqslant I_{\varphi,\psi,g}(f).
 \]

\begin{theorem}[\cite{li-meng-ning-zhou2025arxiv}]
\label{thm:strong-converse}
Let $D\subset\mathbb C^n$ be bounded and $\varphi\in C^2(D)$.  If there exists some
non-zero $f\in A^2(D,\varphi)$ that is strongly $L^2$ divisible with
respect to $\varphi$, then $\varphi$ is plurisubharmonic on $D$.
\end{theorem}

If one chooses $\psi = p (1 + \varepsilon) \log | g |^2$, then it follows from Skoda's Lemme
  Fondamental {\cite{skoda1972}} that the strongly $L^2$ divisibility  implies the Skoda divisibility. However, on bounded pseudoconvex domains, these two properties are both equivalent to the plurisubharmonicity by Theorem \ref{thm:main-inverse-general} and \ref{thm:strong-converse}.

If \(g=(g_1,\dots,g_m)\) has no common zero, then \(g\) induces a surjective holomorphic bundle morphism from the trivial vector bundle \(\underline{\mathbb C}^{\,m}\) to the trivial line bundle \(\underline{\mathbb C}\). In this case, the division theorem can be equivalently reformulated as the surjectivity of the induced map on sections. More generally, one can consider the $L^2$ division problem in the setting of vector bundles.

\begin{theorem}[Skoda \cite{skoda1978}, see also \cite{demailly-bigbook}]\label{thm:skoda-vector-bundle}
  Let $X$ be a Stein manifold, $g:E\rightarrow Q$  a surjective morphism of
  hermitian holomorphic vector bundles and $(L,h)\rightarrow X$ a hermitian holomorphic
  line bundle. Set $n=\dim X$, $r=\operatorname{rk}E$, $q=\operatorname{rk}Q$ and $p=\min\{n,r-q\}$ and
  assume that $E$ is Nakano semipositive and
  \[
  i\Theta_L-p(1+\varepsilon)i\Theta_{\det Q}\geqslant 0
  \]
  for some $\varepsilon>0$. Here $g^\star$ is the metric adjoint, and all norms use the given bundle metrics and a fixed K{\"a}hler metric on $X$. Then for every $F\in H^0(X,Q\otimes L\otimes K_X)$ with
  \[
  J_{h,g,\varepsilon}(F):=\int_X \left\langle (g g^\star)^{-1}F,F\right\rangle_h
  (\det g g^\star)^{-p(1+\varepsilon)} dV<\infty,
  \]
  there exists $H\in H^0(X,E\otimes L\otimes K_X)$ such that $g H=F$ and
  \[
  \int_X |H|_h^2(\det g g^\star)^{-p(1+\varepsilon)} dV\leqslant \left(1+\frac{1}{\varepsilon}\right)J_{h,g,\varepsilon}(F).
  \]
\end{theorem}

We now consider the converse to Theorem \ref{thm:skoda-vector-bundle}. We treat the trivial
bundles with standard hermitian metrics. Let
$
 \underline{\mathbb C}^{m},
 \underline{\mathbb C}^{q}, m>q,
$
be trivial vector bundles over $D$.
Let $G:\underline{\mathbb{C}}^m\to \underline{\mathbb{C}}^q$ be holomorphic surjective bundle
morphism. Put
\begin{equation}\label{equ:matrix-basic-data}
 p=\min\{n,m-q\}.
\end{equation}
For $\varepsilon>0$, define
\begin{equation*}
 \psi_{G,\varepsilon}=p(1+\varepsilon)\log\det(GG^\star)
\end{equation*}
and for $F\in\mathcal O(D)^{\oplus q}$, set
\begin{equation*}
 J_{\varphi,G,\varepsilon}(F)
 =\int_D\left\langle (GG^\star)^{-1}F,F\right\rangle
 e^{-\varphi-\psi_{G,\varepsilon}} dV.
\end{equation*}
When $q=1$, these are precisely \eqref{equ:skoda-weight} and
\eqref{equ:skoda-datum}.

\begin{definition}
\label{def:rank-q-property}
We say that $F\in\mathcal O(D)^{\oplus q}$ is \emph{Skoda
divisible} with respect to $\varphi$ if for every $m>q$,
every pointwise surjective
$G\in\mathcal O(D,\operatorname{Hom}(\underline{\mathbb C}^m,\underline{\mathbb C}^q))$, and every
$\varepsilon>0$ with $J_{\varphi,G,\varepsilon}(F)<\infty$, there is
$H\in\mathcal O(D)^{\oplus m}$ such that $GH=F$ and
\begin{equation*}
 \int_D|H|^2e^{-\varphi-\psi_{G,\varepsilon}} dV
 \leqslant\left(1+\frac1\varepsilon\right)
 J_{\varphi,G,\varepsilon}(F).
\end{equation*}
\end{definition}

We write
\[
 A^2(D,\underline{\mathbb C}^q,\varphi)
 =\left\{F\in\mathcal O(D)^{\oplus q}:
 \int_D|F|^2e^{-\varphi} dV<\infty\right\}.
\]
The following result is the matrix-valued counterpart of Theorem
\ref{thm:main-inverse-general}.

\begin{theorem}
\label{thm:flat-vector-converse}
Let $D\subset\mathbb C^n$ be a bounded domain,
$\varphi\in C^2(D)$, and fix $q\geqslant1$.  Suppose that there is a
non-zero $F\in A^2(D,\underline{\mathbb C}^q,\varphi)$ that is Skoda
divisible with respect to $\varphi$, then
$\varphi$ is plurisubharmonic on $D$.

If in addition, $D$ is pseudoconvex and
$A^2(D,\underline{\mathbb C}^q,\varphi)$ is non-trivial, then the following conditions
are equivalent:
\begin{enumeratenumeric}
 \item $\varphi$ is plurisubharmonic on $D$;
 \item every $F\in\mathcal O(D)^{\oplus q}$ is Skoda
 divisible;
 \item some non-zero $F\in A^2(D,\underline{\mathbb C}^q,\varphi)$ is
 Skoda divisible.
\end{enumeratenumeric}
\end{theorem}

The paper is organized as follows.  Section \ref{sec:preliminaries} recalls some
hermitian vector bundle notation, the K{\"a}hler identities, the
Bochner--Kodaira--Nakano identity, and the Gauss--Codazzi formula. Section \ref{sec:inverse-division}
is devoted to proving Theorem \ref{thm:main-inverse-general}.  Section \ref{sec:higher-rank} proves Theorem \ref{thm:flat-vector-converse}.
Finally, Section \ref{sec:optimal-extension} gives a direct proof  of
the optimal $L^2$ extension.

\section{Preliminaries}\label{sec:preliminaries}

\subsection{Foundations for hermitian holomorphic vector bundles}

We fix the differential-geometric conventions used in the division
argument.  The material in this section is standard, the reader may refer to
\cite{demailly-bigbook,kobayashi-diff-geo-complex-vs}.  It is recalled in
some detail because the signs of the second fundamental form and of the
curvature commutator enter directly into the definition of
$B_{\psi,g}$.

Let $(X,\omega)$ be a K{\"a}hler manifold of complex dimension $n$ and
$(E,h)\to X$ a hermitian holomorphic vector bundle of rank $r$.  The
metric on $E$ and the metric induced by $\omega$ on differential forms give
a pointwise hermitian product, denoted by $\langle\cdot,\cdot\rangle$. If $u$ and $v$ are $E$-valued forms,
write
\[
 \llangle u,v\rrangle_h=\int_X\langle u,v\rangle_h dV_\omega,
 \qquad
 \|u\|_h^2=\llangle u,u\rrangle_h,
\]
we will also write $\llangle u,v\rrangle$ and $\|u\|^2$ when the metric
is clear.
This convention will also be used for sections of a holomorphic subbundle
equipped with its induced metric.

The Chern connection of $(E,h)$ is written
\[
 D_E=D'_E+\bar\partial_E,
\]
where $D'_E$ has bidegree $(1,0)$ and the $(0,1)$ part is the holomorphic
structure.  Its curvature
$\Theta(E)=D_E^2$ is an $\operatorname{End}E$-valued (1,1)-form.  At a point
$x\in X$, choose holomorphic coordinates that are orthonormal for $\omega$
at $x$ and an orthonormal frame $(e_1,\ldots,e_r)$ of $E_x$.  We write
\[
 i\Theta(E)\big|_x
 =i\sum_{j,k=1}^n\sum_{\lambda,\mu=1}^r
 c_{j\bar k\lambda\bar\mu}
 dz_j\wedge d\bar z_k\otimes e_\lambda^*\otimes e_\mu.
\]
The associated hermitian form on $T^{1,0}_xX\otimes E_x$ is
\[
 (\xi_\lambda^j)\longmapsto
 \sum_{j,k,\lambda,\mu}
 c_{j\bar k\lambda\bar\mu}
 \xi_\lambda^j\overline{\xi_\mu^k}.
\]
The bundle $(E,h)$ is Nakano semipositive if this form is non-negative at
every point.  We use the same terminology for a curvature operator on a
specified bundle of forms.

Let $L=L_\omega$ be multiplication by $\omega$ and let $\Lambda=L^*$,
the formal adjoint of $L$.
The formal adjoints of $D'_E$ and $\bar\partial_E$ are denoted by
$D'_E{}^*$ and $\bar\partial_E^*$.  The K{\"a}hler identities are
\begin{equation}\label{equ:kahler-identities}
 [\bar\partial_E^*,L]=iD'_E,
 \quad [D'_E{}^*,L]=-i\bar\partial_E,
 \quad [\Lambda,\bar\partial_E]=-iD'_E{}^*,
 \quad [\Lambda,D'_E]=i\bar\partial_E^*.
\end{equation}

Define the two Laplace operators by
\[
 \Delta'_E=D'_ED'_E{}^*+D'_E{}^*D'_E,
 \qquad
 \Delta''_E=\bar\partial_E\bar\partial_E^*
             +\bar\partial_E^*\bar\partial_E.
\]
The Bochner--Kodaira--Nakano identity states that
\begin{equation}\label{equ:bkn-identity}
 \Delta''_E=\Delta'_E+[i\Theta(E),\Lambda].
\end{equation}
For a compactly supported smooth $E$-valued form $u$, the Bochner--Kodaira--Nakano identity
\eqref{equ:bkn-identity} gives
\begin{equation*}
 \|\bar\partial_Eu\|^2+\|\bar\partial_E^*u\|^2
 =\|D'_Eu\|^2+\|D'_E{}^*u\|^2
  +\llangle[i\Theta(E),\Lambda]u,u\rrangle.
\end{equation*}
For an $E$-valued $(n,1)$-form
\[
 u=\sum_{j=1}^n\sum_{\lambda=1}^r
 u_{j\lambda} dz\wedge d\bar z_j\otimes e_\lambda,
 \qquad dz=dz_1\wedge\cdots\wedge dz_n,
\]
the curvature term at the chosen point is
\begin{equation}\label{equ:curvature-operator-general}
 \big\langle[i\Theta(E),\Lambda]u,u\big\rangle
 =\sum_{j,k,\lambda,\mu}
 c_{j\bar k\lambda\bar\mu}
 u_{j\lambda}\overline{u_{k\mu}}.
\end{equation}
Consequently, Nakano positivity of $(E,h)$ is equivalent to positivity of
$[i\Theta(E),\Lambda]$ on $E$-valued $(n,1)$-forms.  Formula
\eqref{equ:curvature-operator-general} is also the convention behind every
curvature contraction appearing below.

If the metric is changed to $he^{-\Phi}$ for a real-valued $C^2$ function
$\Phi$, then
\begin{equation}\label{equ:weighted-curvature}
 i\Theta(E,he^{-\Phi})
 =i\Theta(E,h)+i\partial\bar\partial\Phi\otimes\mathrm{Id}_E.
\end{equation}
Thus a scalar plurisubharmonic weight contributes a non-negative term to
the curvature operator.

\subsection{The Gauss--Codazzi formula}

Consider an exact sequence of holomorphic vector bundles
\begin{equation*}
 0\longrightarrow S\xrightarrow{j}E\xrightarrow{g}Q\longrightarrow0.
\end{equation*}
Assume that $E$ is hermitian and give $S$ and $Q$ the induced metrics.  Let
$j^*:E\to S$ be orthogonal projection and let $g^*:Q\to E$ be the
orthogonal right inverse.  Then $j^*j=\mathrm{Id}_S$,
$gg^*=\mathrm{Id}_Q$, and the map $j^*\oplus g$ identifies $E$ smoothly
with $S\oplus Q$.

Let $D_E$, $D_S$, and $D_Q$ denote the Chern connections.  Relative to the
$C^\infty$ orthogonal splitting $E\simeq S\oplus Q$, the connection has the
following form,
\begin{equation}\label{equ:gauss-codazzi-connection}
 D_E=
 \begin{bmatrix}
  D_S&-\beta^*\\
  \beta&D_Q
 \end{bmatrix},
\end{equation}
where the
\emph{second fundamental form} of $S$ in $E$
$\beta$ is a $(1,0)$ form valued in $\operatorname{Hom}(S,Q)$. Denote $\beta^*$ by its
adjoint.  The types of the off-diagonal entries imply
\begin{equation}\label{equ:gauss-codazzi-first-order}
 \bar\partial g^*=-\beta^*,
 \qquad
 \bar\partial_S=\bar\partial_E|_S,
 \qquad
 D'_S=D'_E|_S-\beta.
\end{equation}
The second equality and the K{\"a}hler identities
\eqref{equ:kahler-identities} also show that
\begin{equation}\label{equ:chern-adjoint}
 D'_S{}^*=D'_E{}^*,
\end{equation}
although
$D'_S$ and $D'_E$ differ by $\beta$. For simplicity, we will use $\partial^*$ to denote both of them.

Squaring \eqref{equ:gauss-codazzi-connection} gives the curvature matrix
\begin{equation*}
 \Theta(E)=
 \begin{bmatrix}
  \Theta(S)-\beta^*\wedge\beta&-D'\beta^*\\
  \bar\partial\beta&\Theta(Q)-\beta\wedge\beta^*
 \end{bmatrix}.
\end{equation*}
In particular, for a smooth section $s$ of $S$,
\begin{equation}\label{equ:gauss-codazzi-curvature}
 \Theta(S)s=\Theta(E)s+\beta^*\wedge\beta s.
\end{equation}
If $\beta=\sum_jdz_j\otimes\beta_j$, then
\[
 \beta^*\wedge\beta
 =-\sum_{j,k}dz_j\wedge d\bar z_k\otimes\beta_k^*\beta_j.
\]
It follows from \eqref{equ:curvature-operator-general} that the commutator
$[i\beta^*\wedge\beta,\Lambda]$ is non-positive on $S$-valued
$(n,1)$-forms.  More explicitly, if
$u=\sum_ju_j dz\wedge d\bar z_j$, then
\begin{equation*}
 \big\langle[i\beta^*\wedge\beta,\Lambda]u,u\big\rangle
 =-\left|\sum_j\beta_ju_j\right|_Q^2\leqslant0.
\end{equation*}
Hence the subbundle curvature is the ambient curvature together with a
Nakano-negative second-fundamental-form contribution.  When
$E=\underline{\mathbb C}^{m}$ carries the metric $e^{-\Phi}$, formulas
\eqref{equ:weighted-curvature} and
\eqref{equ:gauss-codazzi-curvature} give
\begin{equation*}
 [i\Theta(S,e^{-\Phi}),\Lambda]
 =\left[i\partial\bar\partial\Phi\otimes
 \mathrm{Id}_{\underline{\mathbb C}^{m}}
 +i\beta^*\wedge\beta,\Lambda\right].
\end{equation*}
This is the operator used in the $L^2$ estimate for the kernel equation.
Notice that its second-fundamental-form term is non-positive.

We now specialize these formulas to the division problem.  Let
$D\subset\mathbb C^n$ be a domain and let the holomorphic tuple
$g=(g_1,\ldots,g_m)\in\mathcal O(D)^{\oplus m}$ have no common zero.  The
associated exact sequence is
\begin{equation*}
 0\longrightarrow S\longrightarrow\underline{\mathbb C}^{m}
 \xrightarrow{\ g\ }\underline{\mathbb C}\longrightarrow0,
 \qquad S=\ker g.
\end{equation*}
The standard metrics on the two trivial bundles induce a metric on $S$,
and the orthogonal right inverse is
\begin{equation}\label{equ:orthogonal-right-inverse}
 g^*(a)=\frac{a}{|g|^2}
 (\overline{g_1},\ldots,\overline{g_m}),
 \qquad a\in\mathbb C.
\end{equation}
Thus $gg^*=\mathrm{Id}$ and $g^*(\underline{\mathbb C})=S^\perp$.  For a
holomorphic function $f$, the section $g^*f$ is the pointwise minimal-norm
lifting of $f$ and satisfies $|g^*f|^2=|f|^2/|g|^2$.  Applying
\eqref{equ:gauss-codazzi-first-order} gives
\begin{equation}\label{equ:right-inverse-defect}
 \bar\partial(g^*f)=-\beta^*f.
\end{equation}
Every smooth lifting of $f$ is uniquely of the form $g^*f+u$ with $u$
taking values in $S$, and it is holomorphic precisely when
$\bar\partial_Su=\beta^*f$.  Pointwise orthogonality gives
\[
 |g^*f+u|^2=\frac{|f|^2}{|g|^2}+|u|^2.
\]

For comparison with the strong division estimate from
\cite{li-meng-ning-zhou2025arxiv}, let
$dz=dz_1\wedge\cdots\wedge dz_n$ and, for an admissible auxiliary weight
$\psi$, set
\[
 B_{\psi,g}
 =\left[i\partial\bar\partial\psi\otimes
 \mathrm{Id}_{\underline{\mathbb C}^{m}}
 +i\beta^*\wedge\beta,\Lambda\right].
\]
The right-hand side of that stronger estimate is
\begin{equation*}
 I_{\varphi,\psi,g}(f)
 =\int_D\left\{\frac1{|g|^2}
 +\left\langle B_{\psi,g}^{-1}\beta^*\wedge dz,
 \beta^*\wedge dz\right\rangle\right\}
 |f|^2e^{-\varphi-\psi} dV.
\end{equation*}

\subsection{The classical Skoda estimate}
\label{sec:classical-division}

Let $D\subset\mathbb C^n$ be a domain and assume the components of holomorphic tuple
$g=(g_1,\ldots,g_m)\in\mathcal O(D)^{\oplus m}$ have no common zero.  We
retain the notation $S=\ker g$, $g^*$, and $\beta$ from the preceding
subsection.  We now spell out how \eqref{equ:right-inverse-defect} turns the
classical division estimate into an estimate for the kernel bundle.

\begin{proposition}
  \label{prop:dbar}
  Let $\varphi\in L^1_{\mathrm{loc}}(D)$,
  $f\in\mathcal O(D)$, $\varepsilon>0$, and $p=\min\{n,m-1\}$.  Assume
  $J_{\varphi,g,\varepsilon}(f)<\infty$.  The following two assertions are
  equivalent:
  \begin{enumerate}[a)]
    \item There exists $h \in \mathcal O (D)^{\oplus m}$ such that $f = g
    \cdot h$ and
    \[ \int_D | h |^2 e^{-\varphi-\psi_{g,\varepsilon}} dV
       \leqslant \left( 1 + \frac{1}{\varepsilon} \right)
       J_{\varphi,g,\varepsilon}(f) . \]
    \item There exists an $S$-valued section $u$ satisfying
    \begin{equation}\label{equ:dbar}
    \bar\partial_S u = \beta^* f
    \end{equation}
    with
    \begin{equation}\label{equ:dbar-esitmates}
      \int_D | u |^2 e^{-\varphi-\psi_{g,\varepsilon}} dV
       \leqslant \frac{1}{\varepsilon} J_{\varphi,g,\varepsilon}(f).
    \end{equation}

  \end{enumerate}
\end{proposition}

\begin{proof}
  The proof of this proposition is essentially contained in \cite{li-meng-ning-zhou2025arxiv},
  for the sake of completeness, we include the proof here. Put
  \[
    g^*f=\frac{f}{|g|^2}(\overline{g_1},\ldots,\overline{g_m}).
  \]
  We have $g\cdot g^*f=f$, while $g^*f$ is pointwise orthogonal to $S$.
  Suppose first that $h$ is a holomorphic divisor of $f$ and define
  $u=h-g^*f$.  Since $g\cdot u=0$, the section $u$ is $S$-valued.  Using
  \eqref{equ:right-inverse-defect} and $\bar\partial h=0$, we obtain
  \[
   \bar\partial_Su=\beta^*f.
  \]
  Conversely, let $u$ be $S$-valued and solve \eqref{equ:dbar} with the $L^2$ norm
  estimates.  Then
  $h=g^*f+u$ satisfies $g\cdot h=f$, and
  \eqref{equ:right-inverse-defect} shows that $\bar\partial h=0$ in the
  distribution sense.  Hence $h$ is holomorphic.

  It remains to compare the norms.  Pointwise orthogonality gives the exact
  identity
  \[
    |h|^2=|u|^2+|g^*f|^2=|u|^2+\frac{|f|^2}{|g|^2}.
  \]
  After multiplication by
  $e^{-\varphi-\psi_{g,\varepsilon}}$ and integration, the fixed last term
  is precisely $J_{\varphi,g,\varepsilon}(f)$. Combining this with \eqref{equ:dbar-esitmates}
  gives the desired estimate for $h$.
\end{proof}

\begin{remark}\label{rmk:two-division-functionals}
Proposition \ref{prop:dbar} concerns the special weight
$\psi_{g,\varepsilon}$ and therefore uses
$J_{\varphi,g,\varepsilon}$.  In the strong estimate from our previous
work, a general auxiliary weight $\psi$ is allowed and the correction term
is measured by $ B_{\psi,g}^{-1}$.
\end{remark}

\section{The converse to the classical Skoda estimate}
\label{sec:inverse-division}

We now prove Theorem \ref{thm:main-inverse-general}.  Unlike the  converse of the strong Skoda division theorem 
from \cite{li-meng-ning-zhou2025arxiv}, the hypothesis here supplies no
inverse-curvature norm and no independent plurisubharmonic weight.  Every
localizing and curvature-testing effect must therefore come from the
original Skoda factor $|g|^{-2p(1+\varepsilon)}$ itself.\\

The argument has three stages.  Complex affine changes of coordinates
preserve the division hypothesis.  A special choice of holomorphic tuple in
$\mathcal O(D)^{\oplus(n+1)}$ then creates a
family of normalized measures concentrating at one point.  Finally, the
Bochner--Kodaira--Nakano identity converts the division estimate into a
lower bound for the trace of the Levi form.  Affine invariance promotes this
trace inequality to semi-positivity in every complex direction.

\subsection{Affine invariance}

We begin with the coordinate invariance of the classical estimate.  The
Jacobian factor occurs on both sides and therefore cancels.  It is important
that the number of generators is unchanged, hence the same argument applies
when the hypothesis is restricted to holomorphic tuples in
$\mathcal O(D)^{\oplus(n+1)}$.

\begin{lemma}
  \label{lem:affine}
  Let $T(z)=Az+b$, where $A\in\mathrm{GL}(n,\mathbb C)$. Set
  $\widetilde D=T(D)$,
  $\widetilde\varphi=\varphi\circ T^{-1}$ and
  $\widetilde f=f\circ T^{-1}$. If $f$ is Skoda
  divisible with respect to $\varphi$ on $D$, then $\widetilde f$ is also Skoda
  divisible with respect to $\tilde{\varphi}$ on $\widetilde D$.
\end{lemma}

\begin{proof}
  Fix a holomorphic test tuple
  $\widetilde g\in\mathcal O(\widetilde D)^{\oplus m}$ without a common
  zero and set
  $g=\widetilde g\circ T\in\mathcal O(D)^{\oplus m}$.  Since
  $dV(Tz)=|\det A|^2dV(z)$ and the functions themselves are transported by
  composition, the change-of-variables formula gives
  \begin{equation*}
   J_{\widetilde\varphi,\widetilde g,\varepsilon}(\widetilde f)
   =|\det A|^2J_{\varphi,g,\varepsilon}(f).
  \end{equation*}
  If the right-hand side is finite, the assumed property of $f$ gives
  $h\in\mathcal O(D)^{\oplus m}$ such that $f=g\cdot h$ and
  \[ \int_D | h |^2 e^{- \varphi - p (1 + \varepsilon) \log | g |^2}  dV
     \leqslant \left( 1 + \frac{1}{\varepsilon} \right)
     J_{\varphi,g,\varepsilon}(f), \]
  Define the holomorphic tuple
  $\widetilde h=h\circ T^{-1}\in\mathcal O(\widetilde D)^{\oplus m}$.  Then
  $\widetilde f=\widetilde g\cdot\widetilde h$ and
  \begin{eqnarray*}
    \int_{\widetilde D} | \widetilde h |^2 e^{- \widetilde\varphi - p (1 +
    \varepsilon) \log | \widetilde g |^2}  dV & = & | \det A |^2 \int_D | h |^2
    e^{- \varphi - p (1 + \varepsilon) \log | g |^2}  dV\\
    & \leqslant & \left(1+\frac{1}{\varepsilon}\right)
    J_{\widetilde\varphi,\widetilde g,\varepsilon}(\widetilde f) .
  \end{eqnarray*}
  Thus every admissible test on $\widetilde D$ pulls back to  $D$, proving the assertion.
\end{proof}

\subsection{A special choice of holomorphic test tuple}

Assume from now on that $0\in D$.  For a parameter $M>0$, consider the
special choice of holomorphic test tuple
\begin{equation*}
 g=(z_1,\ldots,z_n,M)\in\mathcal O(D)^{\oplus(n+1)},
 \qquad |g|^2=M^2+|z|^2.
\end{equation*}
Its last component prevents common zeros, while its first $n$ components
retain all coordinate directions.  Let $S=\ker g$ and let $\beta$ be its
second fundamental form.  The next computation records the quantities
needed later in the Bochner--Kodaira--Nakano identity.

\begin{lemma}
  \label{lem:2nd-form}
  Write
  \[
    i\partial\bar\partial\log|g|^2
    =i\sum_{k,l=1}^n H_{k\bar l} dz_k\wedge d\bar z_l,
    \qquad H=(H_{k\bar l}).
  \]
  Then
  \begin{enumeratenumeric}
    \item $|\beta^*|^2=\dfrac{\operatorname{tr}H}{|g|^2}$;

    \item $\left\langle \left[ i \partial\bar\partial \log | g |^2 \otimes
    \mathrm{Id}_{\underline{\mathbb C}^{n + 1}}, \Lambda \right] \beta^*
    \wedge dz, \beta^* \wedge dz \right\rangle = \dfrac{\operatorname{tr} H^2}{| g
    |^2}$;

    \item
    \[ n \left\langle \left[ i \partial\bar\partial \log | g |^2 \otimes
       \mathrm{Id}_{\underline{\mathbb C}^{n + 1}}, \Lambda \right]
       \beta^* \wedge dz, \beta^* \wedge dz \right\rangle - | g |^2 |
       \beta^* |^4 = \frac{(n - 1) | z |^4}{| g |^{10}} . \]
  \end{enumeratenumeric}
\end{lemma}

\begin{proof}
  1. Write $\beta^*=(\beta_1^*,\ldots,\beta_{n+1}^*)$.  The right
    inverse in \eqref{equ:orthogonal-right-inverse} is
    $(\bar z_1,\ldots,\bar z_n,M)/|g|^2$.  Since
    $\beta^*=-\bar\partial g^*$, differentiation gives, for
    $1\leqslant j\leqslant n$,
    \[ \beta^*_j = - \bar\partial \left( \frac{\bar z_j}{| g |^2}
       \right) = - \sum_{k = 1}^n \left( \frac{\delta_{j k}}{| g |^2} -
       \frac{\bar z_j z_k}{| g |^4} \right) d\bar z_k, \]
    while the last component is
    \[ \beta^*_{n + 1} = - \bar\partial \left( \frac{M}{| g |^2}
       \right) = M \sum_{k = 1}^n \frac{z_k}{| g |^4} d\bar z_k . \]
    Decompose the resulting $(0,1)$-form as
    \[ \beta^* = (\beta^*_1, \ldots, \beta^*_{n + 1}) = v_1 d\bar z_1 + \cdots + v_n d\bar z_n, \]
    it follows that
    \begin{equation}
      v_k = - \frac{e_k}{| g |^2} + \frac{x z_k}{| g |^4}
      \label{equ:2nd-coefficients}
    \end{equation}
    where $e_k$ is the $k$-th standard basis vector in
    $\mathbb C^{n+1}$ and
    $x=(\bar z_1,\ldots,\bar z_n,M)$.  Using $|x|^2=|g|^2$, we obtain
    \begin{eqnarray*}
      \langle v_k, v_l \rangle & = & \frac{\delta_{k l}}{| g |^4} + \frac{| x
      |^2 z_k \bar z_l}{| g |^8} - \frac{z_k \bar z_l}{| g |^6} - \frac{z_k
      \bar z_l}{| g |^6}\\
      & = & \frac{\delta_{k l}}{| g |^4} - \frac{z_k \bar z_l}{| g |^6} .
    \end{eqnarray*}
    A separate differentiation of $\log(M^2+|z|^2)$ yields
    \begin{equation}
      (i \partial\bar\partial \log | g |^2)_{k \bar{l}} = \frac{\delta_{k
      l}}{| g |^2} - \frac{z_l \bar z_k}{| g |^4}, \label{equ:curvature}
    \end{equation}
    Comparing the last two matrices, with the harmless transpose forced by
    the form convention, gives
    \[ (\langle v_k, v_l \rangle)_{l k} = \frac{i \partial\bar\partial \log
       | g |^2}{| g |^2} \]
    Taking the trace proves the first identity.
    \[ | \beta^* |^2 = \sum_{k = 1}^n | v_k |^2 = \dfrac{\operatorname{tr} i
       \partial\bar\partial \log | g |^2}{| g |^2} . \]

  2. Formula \eqref{equ:curvature-operator-general}, applied to
    $\beta^*\wedge dz=\sum_kv_kd\bar z_k\wedge dz$, gives
    \begin{eqnarray*}
      &  & \left\langle \left[ i \partial\bar\partial \log | g |^2 \otimes
      \mathrm{Id}_{\underline{\mathbb C}^{n + 1}}, \Lambda \right] \beta^* \wedge dz, \beta^* \wedge dz \right\rangle\\
      & = & \sum_{k, l = 1}^n \left( \frac{\delta_{k l}}{| g |^2} - \frac{z_l
      \bar z_k}{| g |^4} \right)
      \left\langle - \frac{e_k}{| g |^2} + \frac{x z_k}{| g |^4},
      - \frac{e_l}{| g |^2} + \frac{x z_l}{| g |^4}\right\rangle\\
      & = & \sum_{k = 1}^n \sum_{l = 1}^n \left(
      \frac{\delta_{k l}}{| g |^2} - \frac{z_l \bar z_k}{| g |^4} \right)
      \langle v_k, v_l \rangle\\
      & = & \frac{1}{| g |^2} \operatorname{tr} H^2 .
    \end{eqnarray*}

  3. The matrix in \eqref{equ:curvature} acts by $1/|g|^2$ on the
    orthogonal complement of the radial vector and by $M^2/|g|^4$ in the
    radial direction.  Hence its eigenvalues are $1/|g|^2$ with
    multiplicity $n-1$ and $M^2/|g|^4$ with multiplicity one.  Substitution
    into the combination in the statement gives
    \begin{eqnarray*}
      &  & n \left\langle \left[ i \partial\bar\partial \log | g |^2
      \otimes \mathrm{Id}_{\underline{\mathbb C}^{n + 1}}, \Lambda \right]
      \beta^* \wedge dz, \beta^* \wedge dz \right\rangle - | g |^2 |
      \beta^* |^4\\
      & = & n \frac{\operatorname{tr} H^2}{| g |^2} - \dfrac{(\operatorname{tr} H)^2}{| g
      |^2}\\
      & = & \frac{1}{| g |^2} (n \operatorname{tr} H^2 - (\operatorname{tr} H)^2)\\
      & = & \frac{(n - 1)}{| g |^2} \left( \frac{1}{| g |^2} - \frac{M^2}{| g
      |^4} \right)^2\\
      & = & \frac{(n - 1) | z |^4}{| g |^{10}} .
    \end{eqnarray*}
\end{proof}
As a direct consequence of Lemma \ref{lem:2nd-form}, we obtain the following corollary, which will be used in the proof of Theorem \ref{thm:main-inverse-general}.
\begin{corollary}
  \label{cor:2nd-fundamental-compute}
  At $0$, we have
  \begin{enumeratenumeric}
    \item $\beta^* |_0 = - \dfrac{1}{M^2} (d\bar z_1, \ldots, d\bar z_n, 0)$;

    \item $| \beta^* |^2_0 = \dfrac{n}{M^4}$;

    \item
    \[ \left\langle \left[ i \partial\bar\partial \log | g |^2 \otimes
       \mathrm{Id}_{\underline{\mathbb C}^{n + 1}}, \Lambda \right] \beta^* \wedge
       dz, \beta^* \wedge dz \right\rangle |_0 = \frac{n}{M^6} ; \]
    \item $\partial^* (\beta^* \wedge dz) |_0 = 0$.
  \end{enumeratenumeric}
\end{corollary}

\begin{proof}
  1. Evaluating \eqref{equ:2nd-coefficients} at the origin eliminates
    its second term and gives
    \[ \beta^* |_0 = \sum_{k = 1}^n v_k d\bar z_k |_0 = -
       \dfrac{1}{M^2} (d\bar z_1, \ldots, d\bar z_n, 0) . \]

  2. At $0$, the matrix $H$ is $M^{-2}\mathrm{Id}$ and
    $|g|^2=M^2$.  The first assertion of Lemma \ref{lem:2nd-form}
    therefore yields
    \[ | \beta^* |^2_0 = \dfrac{n}{M^4} . \]

  3. The second assertion of Lemma \ref{lem:2nd-form}, together with
    $H(0)=M^{-2}\mathrm{Id}$, gives
    \begin{eqnarray*}
      &  & \left\langle \left[ i \partial\bar\partial \log | g |^2 \otimes
      \mathrm{Id}_{\underline{\mathbb C}^{n + 1}}, \Lambda \right] \beta^* \wedge dz, \beta^* \wedge dz \right\rangle |_0\\
      & = & \dfrac{\operatorname{tr} H^2}{| g |^2} |_0\\
      & = & \frac{n}{M^6} .
    \end{eqnarray*}

  4. The coefficients of $\partial^*(\beta^*\wedge dz)$ are linear
    combinations of the first derivatives of the coefficients of
    $\beta^*$.  For $1\leqslant j\leqslant n$, these derivatives are
    \[ \frac{\partial}{\partial \bar z_s} \left( \frac{\delta_{j k}}{| g |^2}
       - \frac{\bar z_j z_k}{| g |^4} \right) = - \frac{\delta_{j k} z_s +
       \delta_{j s} z_k}{| g |^4} + \frac{2 \bar z_j z_k z_s}{| g |^6} \]
    and
    \[ \frac{\partial}{\partial \bar z_s} \left( M \frac{z_k}{| g |^4}
       \right) = - \frac{  2 M z_k z_s}{| g |^6} . \]
    Each expression contains a factor $z_k$ or $z_s$ and therefore vanishes
    at the origin.  Consequently,
    \[ \partial^* (\beta^* \wedge dz) |_0 = 0. \]
\end{proof}

The following proposition is the key estimate that will be used in the proof of Theorem \ref{thm:main-inverse-general}.

\begin{proposition}
  \label{prop:key-estimate}
  Let $D$ be bounded and assume $0\in D$. Given the holomorphic tuple
  $g=(z_1,\ldots,z_n,M)\in\mathcal O(D)^{\oplus(n+1)}$ with $M>0$.
  Suppose that
  $F\colon D\to[0,\infty)$ is continuous and integrable and
  $F(0)>0$.  Choose $\chi\in C_c^\infty(D)$ with
  $0\leqslant\chi\leqslant1$ and $\chi=1$ on a neighborhood of $0$.  For
  $\varepsilon>0$, define
  \[ d \mu_{\varepsilon} = F e^{- n (1 + \varepsilon) \log | g |^2}  dV. \]
  Then
  \[
   \nu_\varepsilon:=\frac{d\mu_{\varepsilon}}{\int_Dd\mu_{\varepsilon}}
  \]
  converge weakly to $\delta_0$ as $\varepsilon\to\infty$.  Moreover,
  \[ \frac{1}{\int_D d \mu_{\varepsilon}} \left( \frac{\varepsilon \left( \int_D \chi |
     \beta^* |^2 d \mu_{\varepsilon} \right)^2}{\int_D e^{- \log | g |^2} d \mu_{\varepsilon}} - n (1
     + \varepsilon) \int_D \chi^2 \left\langle \left[ i \partial
     \bar\partial \log | g |^2 \otimes \mathrm{Id}_{\underline{\mathbb C}^{n
     + 1}}, \Lambda \right] \beta^* \wedge dz, \beta^* \wedge dz
     \right\rangle d \mu_{\varepsilon} \right) \]
  converges to $-n^2/M^6$ as $\varepsilon\to\infty$.
\end{proposition}

\begin{proof}
  1. We first prove the weak convergence.  Choose $\rho>0$ so that
    $B(0,2\rho)\Subset D$ and $\chi=1$ there. For all sufficiently large
    $\varepsilon$, one has    \[
      B\left(0,\frac{M}{\sqrt{n(1+\varepsilon)}}\right)\subset B(0,\rho).
    \]
    On that ball, $|z|^2/M^2\leqslant1/[n(1+\varepsilon)]$.  It follows from
     $(1+1/q)^q\leqslant e$ that
    \[ \frac{1}{| g |^{2 n (1 + \varepsilon)}} \geqslant \frac{1}{M^{2 n (1 +
       \varepsilon)}} \left( 1 + \frac{1}{n (1 + \varepsilon)} \right)^{- n (1
       + \varepsilon)} \geqslant \frac{1}{e M^{2 n (1 + \varepsilon)}} . \]
    By shrinking $B(0,\rho)$ if necessary, there exist
    constants $c_0,C_0>0$ such that
    $c_0\leqslant F\leqslant C_0$ on $B(0,\rho)$.  Therefore,
    \begin{eqnarray}
      \int_D d \mu_{\varepsilon} & \geqslant & \int_{B \left( 0, \frac{M}{\sqrt{n (1 +
      \varepsilon)}} \right)} F e^{- n (1 + \varepsilon) \log | g |^2}  dV
      \nonumber\\
      & \geqslant & c_0 \int_{B \left( 0, \frac{M}{\sqrt{n (1 +
      \varepsilon)}} \right)} \frac{1}{| g |^{2 n (1 + \varepsilon)}}  dV
      \label{equ:estimate}\\
      & \geqslant & \frac{c_0}{e M^{2 n (1 + \varepsilon)}} \sigma_{2 n}
      \frac{M^{2 n}}{n^n (1 + \varepsilon)^n} \nonumber\\
      & := & c_M \frac{1}{M^{2 n (1 + \varepsilon)} n^n (1 +
      \varepsilon)^n} \nonumber
    \end{eqnarray}
    where $\sigma_{2n}$ denotes the Euclidean volume of the unit ball in
    $\mathbb R^{2n}$ and $c_M>0$ is independent of $\varepsilon$.

    On the other hand,
    \begin{equation*}
      \int_{D\backslash B (0, \rho)} d \mu_{\varepsilon} \leqslant \frac{1}{(M^2 +
      \rho^2)^{n (1 + \varepsilon)}} \int_D F  dV.
    \end{equation*}
    It then follows from \eqref{equ:estimate} that
    \begin{eqnarray}
      &  & \frac{1}{\int_D d \mu_{\varepsilon}} \int_{D\backslash B (0, \rho)} d \mu_{\varepsilon}
      \nonumber\\
      & \leqslant & \frac{n^n (1 + \varepsilon)^n}{c_M} \left( \frac{M^2}{M^2
      + \rho^2} \right)^{n (1 + \varepsilon)} \int_D F  dV
      \label{equ:exponent}\\
      & \rightarrow & 0, \text{ as } \varepsilon\to \infty.\nonumber
    \end{eqnarray}

    Let $\zeta$ be a test function.  Given $\delta>0$,
    choose $r<\rho$ with
    $|\zeta(z)-\zeta(0)|<\delta$ on $B(0,r)$ and therefore,
    \begin{eqnarray*}
      &  & \left| \frac{1}{\int_D d \mu_{\varepsilon}} \int_D \zeta d \mu_{\varepsilon} - \zeta (0)
      \right|\\
      & = & \frac{1}{\int_D d \mu_{\varepsilon}} \left| \int_D \zeta d \mu_{\varepsilon} - \int_D \zeta
      (0) d \mu_{\varepsilon} \right|\\
      & \leqslant & \frac{1}{\int_D d \mu_{\varepsilon}} \int_D | \zeta - \zeta (0) | d
      \mu_{\varepsilon}\\
      & = & \frac{1}{\int_D d \mu_{\varepsilon}} \left( \int_{B (0, r)} + \int_{D\backslash
      B (0, r)} \right) | \zeta - \zeta (0) | d \mu_{\varepsilon}\\
      & < & \delta + \frac{1}{\int_D d \mu_{\varepsilon}} \int_{D\backslash B (0, r)} |
      \zeta - \zeta (0) | d \mu_{\varepsilon},
    \end{eqnarray*}
    Since $|\zeta-\zeta(0)|$ is bounded, repeating
    \eqref{equ:exponent} with $r$ in place of $\rho$ therefore shows that
    \[ \frac{1}{\int_D d \mu_{\varepsilon}} \int_{D\backslash B (0, r)} | \zeta - \zeta (0)
       | d \mu_{\varepsilon} \rightarrow 0, \]
    which implies
    \[ \frac{1}{\int_D d \mu_{\varepsilon}} d \mu_{\varepsilon} \rightarrow \delta_0 \]
    weakly as $\varepsilon \rightarrow \infty$.

  2. By the weak convergence just proved and Corollary
    \ref{cor:2nd-fundamental-compute}, it follows that
    \begin{equation}
      - n \frac{\int_D \chi^2\left\langle \left[ i \partial\bar\partial \log |
      g |^2 \otimes \mathrm{Id}_{\underline{\mathbb C}^{n + 1}}, \Lambda
      \right] \beta^* \wedge dz, \beta^* \wedge dz \right\rangle d \mu_{\varepsilon}}{\int_D d
      \mu_{\varepsilon}} \rightarrow - \frac{n^2}{M^6}. \label{equ:limit2}
    \end{equation}
    Therefore, it suffices to show that
    \begin{equation}
      \frac{- \varepsilon}{\int_D d \mu_{\varepsilon}} \left( n \int_D \chi^2\left\langle \left[ i
      \partial\bar\partial \log | g |^2 \otimes
      \mathrm{Id}_{\underline{\mathbb C}^{n + 1}}, \Lambda \right]
      \beta^* \wedge dz, \beta^* \wedge dz \right\rangle d \mu_{\varepsilon} - \frac{\left(
      \int_D \chi | \beta^* |^2 d \mu_{\varepsilon} \right)^2}{\int_D e^{- \log | g
      |^2} d \mu_{\varepsilon}} \right) \label{equ:limit1}
    \end{equation}
    tends to zero.  Once this is established, adding
    \eqref{equ:limit1} and \eqref{equ:limit2} gives the claimed limit.

    For $t\in\mathbb R$, set
    \[ S (t) = \int_D | \chi | g |^2 | \beta^* |^2 - t |^2 \frac{d \mu_{\varepsilon}}{|
       g |^2}, \]
    this is well-defined because all coefficients involving
    $g$ and $\beta$ are bounded on the bounded domain.  Since
    \[ S (t) = \int_D \chi^2 | g |^4 | \beta^* |^4 \frac{d \mu_{\varepsilon}}{| g |^2}
       - 2 t \int_D \chi | g |^2 | \beta^* |^2 \frac{d \mu_{\varepsilon}}{| g |^2} +
       t^2 \int_D \frac{d \mu_{\varepsilon}}{{| g |^2} } \]
    attains its minimum at
    \[
     t_0=\frac{\displaystyle\int_D\chi|g|^2|\beta^*|^2
     \frac{d\mu_{\varepsilon}}{|g|^2}}
     {\displaystyle\int_De^{-\log|g|^2} d\mu_{\varepsilon}},
    \]
    we have in particular that $S(t_0)\leqslant S((|g|^2|\beta^*|^2)(0))$, that is
    \begin{eqnarray*}
      &  & \int_D | \chi | g |^2 | \beta^* |^2 - (| g |^2 | \beta^*
      |^2) (0)  |^2 \frac{d \mu_{\varepsilon}}{| g |^2}\\
      & = & S ((| g |^2 | \beta^* |^2) (0))\\
      & \geqslant & \int_D \chi^2 | g |^4 | \beta^* |^4 \frac{d \mu_{\varepsilon}}{| g
      |^2} - \frac{\left( \int_D \chi | g |^2 | \beta^* |^2 \frac{d
      \mu_{\varepsilon}}{| g |^2} \right)^2}{\int_D e^{- \log | g |^2} d \mu_{\varepsilon}}\\
      & = & \int_D \chi^2 | g |^2 | \beta^* |^4 d \mu_{\varepsilon} - \frac{\left(
      \int_D \chi | \beta^* |^2 d \mu_{\varepsilon} \right)^2}{\int_D e^{- \log | g
      |^2} d \mu_{\varepsilon}} .
    \end{eqnarray*}
    Since $\chi=1$ near $0$, by shrinking $B(0,\rho)$ if necessary, there exists a constant $c_1$ such that
    \[
     \big|\chi|g|^2|\beta^*|^2-(|g|^2|\beta^*|^2)(0)\big|
     \leqslant c_1|z|^2
    \]
    on $B(0,\rho)$.  Note that there exists $c_2 > 0$ such that $| \chi | g |^2
    | \beta^{\ast} |^2 - (| g |^2 | \beta^{\ast} |^2) (0)  |^2 \leqslant c_2$ on
    $D$, it then follows that
    \begin{eqnarray*}
      &  & \frac{1}{\int_D d \mu_{\varepsilon}} \int_D | \chi | g |^2 | \beta^* |^2 -
      (| g |^2 | \beta^* |^2) (0)  |^2 \frac{d \mu_{\varepsilon}}{| g |^2}\\
      & = & \frac{1}{\int_D d \mu_{\varepsilon}} \left( \int_{B (0, \rho)} +
      \int_{D\backslash B (0, \rho)} \right) | \chi | g |^2 | \beta^* |^2
      - (| g |^2 | \beta^* |^2) (0)  |^2 \frac{d \mu_{\varepsilon}}{| g |^2}\\
      & \leqslant & \frac{c_1}{M^2 \int_D d \mu_{\varepsilon}} \int_{B (0, \rho)} | z |^4 d
      \mu_{\varepsilon} + \frac{1}{\int_D d \mu_{\varepsilon}} \int_{D\backslash B (0, \rho)} | \chi | g
      |^2 | \beta^* |^2 - (| g |^2 | \beta^* |^2) (0)  |^2 \frac{d
      \mu_{\varepsilon}}{| g |^2}\\
      & = & A + B.
    \end{eqnarray*}
    We now estimate $A$ and $B$ separately.
    For $k=1,2$ and $n(1+\varepsilon)>n+k$, polar coordinates followed by
    the substitution $s=r^2/M^2$ give
    \begin{eqnarray*}
      &  & \int_{\mathbb C^n} | z |^{2 k} \frac{1}{| g |^{2 n (1 +
      \varepsilon)}}  dV\\
      & = & c_{2 n - 1} \int_0^{\infty} r^{2 k + 2 n - 1} \frac{1}{(M^2 +
      r^2)^{n (1 + \varepsilon)}} d r\\
      & = & \frac{c_{2 n - 1}}{2} M^{2 (n + k - n (1 + \varepsilon))} \Beta
      (n + k, n (1 + \varepsilon) - n - k)
    \end{eqnarray*}
    where $c_{2n-1}$ is the volume of the unit sphere in $\mathbb R^{2n}$ and
    $\Beta$ is the Beta function.  The identities
    \[ \Beta (n + k, n (1 + \varepsilon) - n - k) = \frac{\Gamma (n + k) \Gamma (n
       (1 + \varepsilon) - n - k)}{\Gamma (n (1 + \varepsilon))}, \]
    and
    \[ \lim_{\varepsilon \rightarrow \infty} \frac{\Gamma (n + k) \Gamma (n (1
       + \varepsilon) - n - k) (n (1 + \varepsilon))^{n + k}}{\Gamma (n (1 +
       \varepsilon))} = \Gamma (n + k), \]
    it then follows that there exists $\tilde{C}'$ such that
    \[
    \int_{\mathbb C^n} | z |^{2 k} \frac{1}{| g |^{2 n (1 +
      \varepsilon)}}  dV\leqslant \tilde{C}' \frac{M^{2 (n + k - n
      (1 + \varepsilon))}}{(n (1 + \varepsilon))^{n + k}}.
    \]
    Combining this with the
    normalization lower bound \eqref{equ:estimate}, we obtain a constant
    $\widetilde C'$ such that
    \begin{eqnarray*}
      \frac{1}{\int_D d \mu_{\varepsilon}} \int_{\mathbb C^n} | z |^{2 k} \frac{1}{| g |^{2
      n (1 + \varepsilon)}}  dV & \leqslant & \tilde{C}' \frac{M^{2 (n + k - n
      (1 + \varepsilon))}}{(n (1 + \varepsilon))^{n + k}} \cdot \frac{M^{2 n
      (1 + \varepsilon)} [n (1 + \varepsilon)]^n}{c_M} \\
      & := & \frac{C''}{(1 + \varepsilon)^k} .
    \end{eqnarray*}
    Hence,
    \begin{eqnarray}
      &  & \frac{1}{\int_D d \mu_{\varepsilon}} \int_D | z |^{2 k} d \mu_{\varepsilon}
      \nonumber\\
      & = & \frac{1}{\int_D d \mu_{\varepsilon}} \left( \int_{B (0, \rho)} +
      \int_{D\backslash B (0, \rho)} \right) | z |^{2 k} d \mu_{\varepsilon} \nonumber\\
      & \leqslant & \frac{C_0}{\int_D d \mu_{\varepsilon}} \int_{\mathbb C^n} \frac{| z
      |^{2 k}}{| g |^{2 n (1 + \varepsilon)}}  dV + \frac{1}{\int_D d \mu_{\varepsilon}}
      \int_{D\backslash B (0, \rho)} | z |^{2 k} d \mu_{\varepsilon}
      \label{equ:estimates2}\\
      & \leqslant & \frac{C''C_0}{(1 + \varepsilon)^k} + \sup_{z \in D} | z |^{2
      k} \cdot \frac{n^n (1 + \varepsilon)^n}{c_M} \left( \frac{M^2}{M^2 +
      \rho^2} \right)^{n (1 + \varepsilon)} \int_D F  dV. \nonumber
    \end{eqnarray}

    Taking $k=2$, there exists $C'$ such that
    \[ A \leqslant \frac{C'}{(1 + \varepsilon)^2} + \sup_{z \in D} | z |^4
       \cdot \frac{n^n (1 + \varepsilon)^n}{c_M} \left( \frac{M^2}{M^2 +
       \rho^2} \right)^{n (1 + \varepsilon)} \int_D F  dV. \]
    It follows directly by
    \eqref{equ:exponent} that
    \begin{eqnarray*}
      B & = & \frac{1}{\int_D d \mu_{\varepsilon}} \int_{D\backslash B (0, \rho)} | \chi | g
      |^2 | \beta^* |^2 - (| g |^2 | \beta^* |^2) (0)  |^2 \frac{d
      \mu_{\varepsilon}}{| g |^2}\\
      & \leqslant & \frac{c_2}{M^2} \frac{1}{\int_D d \mu_{\varepsilon}} \int_{D\backslash
      B (0, \rho)} d \mu_{\varepsilon}\\
      & \leqslant & \frac{c_2}{M^2} \frac{n^n (1 + \varepsilon)^n}{c_M}
      \left( \frac{M^2}{M^2 + \rho^2} \right)^{n (1 + \varepsilon)} \int_D F  dV.
    \end{eqnarray*}
    By the Cauchy--Schwarz inequality, the minimizing property of $S(t)$, and the last two
    estimates, we have
    \begin{eqnarray*}
      & 0 \leqslant & \frac{1}{\int_D d \mu_{\varepsilon}} \left( \int_D \chi^2 | g |^2 |
      \beta^* |^4 d \mu_{\varepsilon} - \frac{\left( \int_D \chi | \beta^* |^2 d
      \mu_{\varepsilon} \right)^2}{\int_D e^{- \log | g |^2} d \mu_{\varepsilon}} \right)\\
      & \leqslant & \frac{1}{\int_D d \mu_{\varepsilon}} \int_D | \chi | g |^2 |
      \beta^* |^2 - (| g |^2 | \beta^* |^2) (0)  |^2 \frac{d \mu_{\varepsilon}}{| g
      |^2}\\
      & \leqslant & \frac{C'}{(1 + \varepsilon)^2} + \sup_{z \in D} | z |^4
      \cdot \frac{n^n (1 + \varepsilon)^n}{c_M} \left( \frac{M^2}{M^2 +
      \rho^2} \right)^{n (1 + \varepsilon)} \int_D F  dV\\
      &  & + \frac{c_2}{M^2} \frac{n^n (1 + \varepsilon)^n}{c_M} \left(
      \frac{M^2}{M^2 + \rho^2} \right)^{n (1 + \varepsilon)} \int_D F  dV.
    \end{eqnarray*}
    On the other hand, it follows from
    Lemma \ref{lem:2nd-form} and
    \eqref{equ:estimates2} that
    \begin{eqnarray*}
      &  & \frac{1}{\int_D d \mu_{\varepsilon}} \left( n \int_D \chi^2 \left\langle \left[
      i \partial\bar\partial \log | g |^2 \otimes
      \mathrm{Id}_{\underline{\mathbb C}^{n + 1}}, \Lambda \right]
      \beta^* \wedge dz, \beta^* \wedge dz \right\rangle d \mu_{\varepsilon} - \int_D \chi^2 |
      g |^2 | \beta^* |^4 d \mu_{\varepsilon} \right)\\
      & = & \frac{1}{\int_D d \mu_{\varepsilon}} \int_D \chi^2
      \frac{(n - 1) | z |^4}{| g |^{10}} d \mu_{\varepsilon}\\
      & \leqslant & \frac{n - 1}{M^{10}} \frac{1}{\int_D d \mu_{\varepsilon}} \int_D | z
      |^4 d \mu_{\varepsilon}\\
      & \leqslant & \frac{n - 1}{M^{10}} \left[ \frac{C'}{(1 +
      \varepsilon)^2} + \sup_{z \in D} | z |^4 \cdot \frac{n^n (1 +
      \varepsilon)^n}{c_M} \left( \frac{M^2}{M^2 + \rho^2} \right)^{n (1 +
      \varepsilon)} \int_D F  dV \right] .
    \end{eqnarray*}
    Combining these two estimates,
    we obtain
    \begin{eqnarray*}
      &  & \frac{\varepsilon}{\int_D d \mu_{\varepsilon}} \left( n \int_D \chi^2\left\langle
      \left[ i \partial\bar\partial \log | g |^2 \otimes
      \mathrm{Id}_{\underline{\mathbb C}^{n + 1}}, \Lambda \right]
      \beta^* \wedge dz, \beta^* \wedge dz \right\rangle d \mu_{\varepsilon} - \frac{\left(
      \int_D \chi | \beta^* |^2 d \mu_{\varepsilon} \right)^2}{\int_D e^{- \log | g
      |^2} d \mu_{\varepsilon}} \right)\\
      & \leqslant & \frac{\varepsilon C'}{(1 + \varepsilon)^2} + \sup_{z \in
      D} | z |^4 \cdot \frac{\varepsilon n^n (1 + \varepsilon)^n}{c_M} \left(
      \frac{M^2}{M^2 + \rho^2} \right)^{n (1 + \varepsilon)} \int_D F  dV\\
      &  &+\varepsilon\frac{c_2}{M^2} \frac{n^n (1 + \varepsilon)^n}{c_M}
      \left( \frac{M^2}{M^2 + \rho^2} \right)^{n (1 + \varepsilon)} \int_D F dV\\
      &  & + \frac{n - 1}{M^{10}} \left[ \frac{\varepsilon C'}{(1 +
      \varepsilon)^2} + \sup_{z \in D} | z |^4 \cdot \frac{\varepsilon n^n
      (1 + \varepsilon)^n}{c_M} \left( \frac{M^2}{M^2 + \rho^2} \right)^{n (1
      + \varepsilon)} \int_D F  dV \right]\\
      & \rightarrow & 0.
    \end{eqnarray*}

\end{proof}

\begin{proposition}\label{rmk:r-generators}
Let $1\leqslant r<n$, write
\[
 z=(z',z'')\in\mathbb C^r\times\mathbb C^{n-r},
 \qquad L=\{z'=0\},
\]
and take the holomorphic tuple
\[
 g=(z_1,\ldots,z_r,M)\in\mathcal O(D)^{\oplus(r+1)},
 \qquad |g|^2=M^2+|z'|^2.
\]
Let $F$ be non-negative, continuous, and integrable on the bounded domain
$D$, and set
\[
 d\mu_\varepsilon
 =F(z',z'')(M^2+|z'|^2)^{-r(1+\varepsilon)}
  dV_{z'} dV_{z''}.
\]
For a bounded continuous function $\zeta$, define
\[
 G_\zeta(z')=
 \int_{\{z'':(z',z'')\in D\}}
 \zeta(z',z'')F(z',z'') dV_{z''}.
\]
Assume that $G_\zeta$ is continuous at $0$ for every such $\zeta$, and that
\[
 0<\int_{D\cap L}F(0,z'') dV_L(z'')<\infty.
\]
Then
\begin{equation}\label{equ:slice-concentration}
 \frac{d\mu_\varepsilon}{\int_Dd\mu_\varepsilon}
 \longrightarrow
 \frac{F|_L dV_L}{\int_{D\cap L}F dV_L}
\end{equation}
weakly as $\varepsilon\to\infty$.  Equivalently, every bounded continuous
function $\zeta$ satisfies
\[
 \lim_{\varepsilon\to\infty}
 \frac{\int_D\zeta d\mu_\varepsilon}{\int_Dd\mu_\varepsilon}
 =\frac{\int_{D\cap L}\zeta(0,z'')F(0,z'') dV_L(z'')}
 {\int_{D\cap L}F(0,z'') dV_L(z'')}.
\]
\end{proposition}

\begin{proof}
Define
\[
 k_\varepsilon(z')=(M^2+|z'|^2)^{-r(1+\varepsilon)},
 \qquad
 c_\varepsilon=\int_{\mathbb C^r}k_\varepsilon(z') dV_{z'}.
\]
Polar coordinates show that $c_\varepsilon$ is finite and positive.  We claim that
$K_\varepsilon=c_\varepsilon^{-1}k_\varepsilon$ is an approximate identity
at the origin.  Indeed, for every $\rho>0$, the same inner-ball and
outer-ball estimates used in Proposition \ref{prop:key-estimate}, with $r$
in place of $n$, give
\begin{equation}\label{equ:normal-approximate-identity}
 \int_{|z'|\geqslant\rho}K_\varepsilon(z') dV_{z'}
 \longrightarrow0.
\end{equation}
By the continuity assumption on the fiber integrals,
\[
 G_\zeta(z')\longrightarrow
 G_\zeta(0)=
 \int_{D\cap L}\zeta(0,z'')F(0,z'') dV_L(z'')
\]
as $z'\to0$.  Moreover, Fubini's theorem gives
\[
 \|G_\zeta\|_{L^1(\mathbb C^r)}
 \leqslant\|\zeta\|_\infty\int_D F\,dV<\infty.
\]
Polar coordinates give
\[
 c_\varepsilon
 =\pi^r M^{-2r\varepsilon}
   \frac{\Gamma(r\varepsilon)}{\Gamma(r(1+\varepsilon))}.
\]
Consequently, for every $\rho>0$,
\[
\begin{aligned}
 \sup_{|z'|\geqslant\rho}K_\varepsilon(z')
 &=\frac{\Gamma(r(1+\varepsilon))}
         {\pi^rM^{2r}\Gamma(r\varepsilon)}
   \left(\frac{M^2}{M^2+\rho^2}\right)^{r(1+\varepsilon)}\\
 &\leqslant\frac{[r(1+\varepsilon)]^r}{\pi^rM^{2r}}
   \left(\frac{M^2}{M^2+\rho^2}\right)^{r(1+\varepsilon)}
 \longrightarrow0,
\end{aligned}
\]
where we used
$\Gamma(r(1+\varepsilon))/\Gamma(r\varepsilon)
=\prod_{j=0}^{r-1}(r\varepsilon+j)$.
It follows that
\[
 \left|\int_{|z'|\geqslant\rho}
 K_\varepsilon(z')G_\zeta(z')\,dV_{z'}\right|
 \leqslant\sup_{|z'|\geqslant\rho}K_\varepsilon(z')
           \|G_\zeta\|_{L^1(\mathbb C^r)}
 \longrightarrow0.
\]
Combining this estimate with continuity at $0$,
\eqref{equ:normal-approximate-identity}, and Fubini's theorem gives
\[
 \frac{1}{c_\varepsilon}\int_D\zeta d\mu_\varepsilon
 =\int_{\mathbb C^r}K_\varepsilon(z')G_\zeta(z') dV_{z'}
 \longrightarrow G_\zeta(0).
\]
Apply the same identity with $\zeta=1$.  Its limit is positive by
hypothesis, so taking the quotient proves
\eqref{equ:slice-concentration}.
\end{proof}

\begin{remark}\label{rmk:slice-obstruction}

When $r<n$, the Skoda weight localizes only the $r$ directions normal to
$L$, the remaining $n-r$ variables survive in the limiting measure.  This
is fundamentally different from Proposition \ref{prop:key-estimate}, where
the holomorphic tuple $g\in\mathcal O(D)^{\oplus(n+1)}$ produces a point
mass.
\end{remark}

\subsection{Proof of the converse}
\label{sec:proof-converse}

We now apply the radial calculation to obtain Theorem \ref{thm:main-inverse-general}.  We first show that under the hypotheses of Theorem \ref{thm:main-inverse-general}, the trace of the complex Hessian of $\varphi$ is non-negative.
Finally, affine invariance of the hypotheses allows us to conclude that the complex Hessian is non-negative definite.
\begin{lemma}
  \label{lem:trace}
  Let $D\subset\mathbb C^n$ be a domain,  $0\in D$ and
  $\varphi\in C^2(D)$.  For the holomorphic tuple
  $g=(z_1,\ldots,z_n,M)\in\mathcal O(D)^{\oplus(n+1)}$, let $\beta$ be the
  second fundamental form of $S=\ker g$.  Then
  \[ \langle [i \partial\bar\partial \varphi\otimes
     \mathrm{Id}_{\underline{\mathbb C}^{n+1}}, \Lambda] \beta^* \wedge
     dz, \beta^* \wedge dz \rangle |_0 = \frac{1}{M^4} \operatorname{tr} i
     \partial\bar\partial \varphi |_0 . \]
\end{lemma}

\begin{proof}
  Write
  \[
   i\partial\bar\partial\varphi
   =i\sum_{j,k}\varphi_{j\bar k} dz_j\wedge d\bar z_k.
  \]
  At the origin, Corollary \ref{cor:2nd-fundamental-compute} gives
  \[
   \beta^*=-M^{-2}(d\bar z_1,\ldots,d\bar z_n,0).
  \]
  It then follows from
  \eqref{equ:curvature-operator-general} that
  \begin{eqnarray*}
    \langle [i \partial\bar\partial \varphi\otimes
    \mathrm{Id}_{\underline{\mathbb C}^{n+1}}, \Lambda] \beta^* \wedge dz, \beta^* \wedge dz \rangle |_0
    & = & \frac{1}{M^4}\sum_{j=1}^n\varphi_{j\bar j}(0)\\
    & = & \frac{1}{M^4} \operatorname{tr} i \partial\bar\partial \varphi |_0 .
  \end{eqnarray*}
\end{proof}

\begin{proposition}
  \label{prop:trace}
  Let $D\subset\mathbb C^n$ be a bounded domain and
  $\varphi\in C^2(D)$.  Suppose that a non-zero
  $f\in A^2(D,\varphi)$ is Skoda divisible for every
  holomorphic tuple $g\in\mathcal O(D)^{\oplus(n+1)}$ without a common zero
  and every $\varepsilon>0$ for
  which $J_{\varphi,g,\varepsilon}(f)<\infty$.  Then $\varphi$ is
  subharmonic on $D$.
\end{proposition}

\begin{proof}
  We must prove
  $\operatorname{tr}i\partial\bar\partial\varphi\geqslant0$.  Suppose to
  the contrary that this trace is negative at $z_0$. Without loss of generality, assume $f(z_0)\ne0$. Thanks to Lemma \ref{lem:affine},
  we may also assume $z_0=0$.

  Fix $M>0$ and choose
  $g=(z_1,\ldots,z_n,M)\in\mathcal O(D)^{\oplus(n+1)}$, here $p=n$. Since $| g |^2 \geqslant
  M^2$, $| g |^{- 2 n (1 + \varepsilon)}$
  and $| g |^{- 2 n (1 + \varepsilon) - 2}$ are bounded for fixed $M$ and
  $\varepsilon > 0$,  set
  $\Phi_{\varepsilon}=\varphi+\psi_{g,\varepsilon}$,
  hence
  \[ \int_D | f |^2 e^{- \Phi_{\varepsilon}}  dV < \infty \]
  and
  \[ J_{\varphi,g,\varepsilon}(f) = \int_D \frac{| f |^2}{| g |^2} e^{- \Phi_{\varepsilon}}  dV < \infty . \]
  It follows from Proposition \ref{prop:dbar} that there exists  an $S$-valued section
  $u_\varepsilon$ satisfying
  \[ \bar\partial_S u_{\varepsilon} = \beta^* f \]
  with
  \begin{equation}
    \int_D | u_{\varepsilon} |^2 e^{- \Phi_{\varepsilon}}  dV \leqslant
    \frac{1}{\varepsilon} J_{\varphi,g,\varepsilon}(f) .
    \label{equ:dbar-estimate}
  \end{equation}
  Choose $\chi\in C_c^\infty(D)$ with $0\leqslant\chi\leqslant1$ and
  $\chi=1$ near $0$.  Passing to top holomorphic degree, define
  \[ \tilde{u}_{\varepsilon} = u_{\varepsilon} \wedge dz, \alpha = \chi
     \beta^* f \wedge dz, \]
  Then
  $\bar\partial_S\widetilde u_\varepsilon=\beta^*f\wedge dz$.  Because
  $\alpha$ is compactly supported, we have
  \begin{eqnarray*}
    \int_D \chi | \beta^* |^2 | f |^2 e^{- \Phi_{\varepsilon}}  dV & = & \llangle
    \bar\partial \tilde{u}_{\varepsilon}, \alpha \rrangle_{\Phi_{\varepsilon}}\\
    & = & \llangle \tilde{u}_{\varepsilon}, \bar\partial^*_S \alpha
    \rrangle_{\Phi_{\varepsilon}} .
  \end{eqnarray*}
  The Cauchy--Schwarz inequality and \eqref{equ:dbar-estimate} imply
  \[ \left( \int_D \chi | \beta^* |^2 | f |^2 e^{- \Phi_{\varepsilon}}  dV \right)^2
     \leqslant \| \tilde{u}_{\varepsilon} \|^2_{\Phi_{\varepsilon}} \| \bar\partial^*_S
     \alpha \|^2_{\Phi_{\varepsilon}} \leqslant \frac{1}{\varepsilon}
     J_{\varphi,g,\varepsilon}(f) \| \bar\partial^*_S \alpha \|^2_{\Phi_{\varepsilon}} \]
  and therefore
  \[ \| \bar\partial^*_S \alpha \|^2_{\Phi_{\varepsilon}} \geqslant
     \frac{\varepsilon}{J_{\varphi,g,\varepsilon}(f)} \left( \int_D \chi |
     \beta^* |^2 | f |^2 e^{- \Phi_{\varepsilon}}  dV \right)^2 . \]
  It follows from the K{\"a}hler identities and \eqref{equ:chern-adjoint} that
  \[ \| D'_S{}^*\alpha \|^2_{\Phi_{\varepsilon}}=\|\partial^*\alpha\|^2_{\Phi_{\varepsilon}}
     = \int_D | f |^2 | \partial^* (\chi \beta^*
     \wedge dz) |^2 e^{- \Phi_{\varepsilon}}  dV. \]
  It then follows from the Bochner-Kodaira equality that
  \begin{eqnarray*}
    &  & \| \bar\partial_S^* \alpha \|^2_{\Phi_{\varepsilon}} + \| \bar\partial
    \alpha \|^2_{\Phi_{\varepsilon}}\\
    & = & \| D'_S{}^*\alpha \|^2_{\Phi_{\varepsilon}} + \int_D \chi^2 \left\langle
    \left[ i \partial\bar\partial \varphi \otimes
    \mathrm{Id}_{\underline{\mathbb C}^{n + 1}}, \Lambda \right] \beta^*
    \wedge dz, \beta^* \wedge dz \right\rangle  | f |^2 e^{- \Phi_{\varepsilon}}  dV\\
    &  & + n (1 + \varepsilon) \int_D \chi^2 \left\langle \left[ i \partial
    \bar\partial \log | g |^2 \otimes \mathrm{Id}_{\underline{\mathbb C}^{n +
    1}}, \Lambda \right] \beta^* \wedge dz, \beta^* \wedge dz
    \right\rangle  | f |^2 e^{- \Phi_{\varepsilon}}  dV\\
    &  & + \int_D \chi^2 \langle [i \beta^* \wedge \beta, \Lambda]
    \beta^* \wedge dz, \beta^* \wedge dz \rangle  | f |^2 e^{-
    \Phi_{\varepsilon}}  dV.
  \end{eqnarray*}
  Thus
  \begin{eqnarray*}
    &  & \frac{\varepsilon}{J_{\varphi,g,\varepsilon}(f)} \left( \int_D \chi |
    \beta^* |^2 | f |^2 e^{- \Phi_{\varepsilon}}  dV \right)^2\\
    & \leqslant & \| \partial^* \alpha \|^2_{\Phi_{\varepsilon}} + \int_D \chi^2
    \left\langle \left[ i \partial\bar\partial \varphi \otimes
    \mathrm{Id}_{\underline{\mathbb C}^{n + 1}}, \Lambda \right] \beta^*
    \wedge dz, \beta^* \wedge dz \right\rangle  | f |^2 e^{- \Phi_{\varepsilon}}  dV\\
    &  & + n (1 + \varepsilon) \int_D \chi^2 \left\langle \left[ i \partial
    \bar\partial \log | g |^2 \otimes \mathrm{Id}_{\underline{\mathbb C}^{n +
    1}}, \Lambda \right] \beta^* \wedge dz, \beta^* \wedge dz
    \right\rangle  | f |^2 e^{- \Phi_{\varepsilon}}  dV\\
    &  & + \int_D \chi^2 \langle [i \beta^* \wedge \beta, \Lambda]
    \beta^* \wedge dz, \beta^* \wedge dz \rangle  | f |^2 e^{-
    \Phi_{\varepsilon}}  dV - \| \bar\partial \alpha \|^2_{\Phi_{\varepsilon}}\\
    & \leqslant & \| \partial^* \alpha \|^2_{\Phi_{\varepsilon}} + \int_D \chi^2
    \left\langle \left[ i \partial\bar\partial \varphi \otimes
    \mathrm{Id}_{\underline{\mathbb C}^{n + 1}}, \Lambda \right] \beta^*
    \wedge dz, \beta^* \wedge dz \right\rangle  | f |^2 e^{- \Phi_{\varepsilon}}  dV\\
    &  & + n (1 + \varepsilon) \int_D \chi^2 \left\langle \left[ i \partial
    \bar\partial \log | g |^2 \otimes \mathrm{Id}_{\underline{\mathbb C}^{n +
    1}}, \Lambda \right] \beta^* \wedge dz, \beta^* \wedge dz
    \right\rangle  | f |^2 e^{- \Phi_{\varepsilon}}  dV,
  \end{eqnarray*}
  that is,
  \begin{eqnarray}
    &  & \int_D \chi^2 \left\langle \left[ i \partial\bar\partial \varphi
    \otimes \mathrm{Id}_{\underline{\mathbb C}^{n + 1}}, \Lambda \right]
    \beta^* \wedge dz, \beta^* \wedge dz \right\rangle  | f |^2
    e^{- \Phi_{\varepsilon}}  dV + \| \partial^* \alpha \|^2_{\Phi_{\varepsilon}} \nonumber\\
    & \geqslant & \frac{\varepsilon}{J_{\varphi,g,\varepsilon}(f)} \left(
    \int_D \chi | \beta^* |^2 | f |^2 e^{- \Phi_{\varepsilon}}  dV \right)^2
    \label{equ:estimate3}\\
    &  & - n (1 + \varepsilon) \int_D \chi^2 \left\langle \left[ i \partial
    \bar\partial \log | g |^2 \otimes \mathrm{Id}_{\underline{\mathbb C}^{n +
    1}}, \Lambda \right] \beta^* \wedge dz, \beta^* \wedge dz
    \right\rangle  | f |^2 e^{- \Phi_{\varepsilon}}  dV. \nonumber
  \end{eqnarray}
  It follows from Proposition \ref{prop:key-estimate} with $F = | f |^2 e^{-
  \varphi}$ and Lemma \ref{lem:trace} that
  \[ \frac{\int_D \chi^2 \left\langle \left[ i \partial\bar\partial \varphi
     \otimes \mathrm{Id}_{\underline{\mathbb C}^{n + 1}}, \Lambda \right]
     \beta^* \wedge dz, \beta^* \wedge dz \right\rangle  | f |^2
     e^{- \Phi_{\varepsilon}}  dV}{\int_D | f |^2 e^{- \Phi_{\varepsilon}}  dV} \rightarrow \frac{1}{M^4}
     \operatorname{tr} i \partial\bar\partial \varphi |_0 \]
  and
  \[ \frac{\| \partial^* \alpha \|^2_{\Phi_{\varepsilon}}}{\int_D | f |^2 e^{- \Phi_{\varepsilon}}  dV} = \frac{\int_D |f|^2 |\partial^* (\chi \beta^* \wedge dz) |^2
     e^{- \Phi_{\varepsilon}}  dV}{\int_D | f |^2 e^{- \Phi_{\varepsilon}}  dV}
     \rightarrow |\partial^* (\chi \beta^* \wedge dz) |^2 |_0 = 0 \]
  as $\varepsilon\to\infty$.  One also has, by \eqref{equ:estimate3},
  that
  \begin{eqnarray*}
    &  & \frac{1}{\int_D | f |^2 e^{- \Phi_{\varepsilon}}  dV} \left( \int_D \chi^2
    \left\langle \left[ i \partial\bar\partial \varphi \otimes
    \mathrm{Id}_{\underline{\mathbb C}^{n + 1}}, \Lambda \right] \beta^*
    \wedge dz, \beta^* \wedge dz \right\rangle  | f |^2 e^{- \Phi_{\varepsilon}}  dV
    + \| \partial^* \alpha \|^2_{\Phi_{\varepsilon}} \right)\\
    & \geqslant & \frac{\varepsilon}{J_{\varphi,g,\varepsilon}(f) \int_D | f
    |^2 e^{- \Phi_{\varepsilon}}  dV} \left( \int_D \chi | \beta^* |^2 | f |^2 e^{-
    \Phi_{\varepsilon}}  dV \right)^2\\
    &  & - \frac{n (1 + \varepsilon)}{\int_D | f |^2 e^{- \Phi_{\varepsilon}}  dV} \int_D
    \chi^2 \left\langle \left[ i \partial\bar\partial \log | g |^2 \otimes
    \mathrm{Id}_{\underline{\mathbb C}^{n + 1}}, \Lambda \right] \beta^*
    \wedge dz, \beta^* \wedge dz \right\rangle  | f |^2 e^{- \Phi_{\varepsilon}}  dV.
  \end{eqnarray*}
  Let $\varepsilon \rightarrow \infty$, it follows from Proposition
  \ref{prop:key-estimate} that
  \[ \operatorname{tr} i \partial\bar\partial \varphi |_0 \geqslant -
     \frac{n^2}{M^2} . \]
  Since the above inequality holds for every $M>0$, letting
  $M\rightarrow\infty$ contradicts
  $\operatorname{tr} i \partial \bar{\partial} \varphi |_0 < 0$. Thus
  $\operatorname{tr} i \partial \bar{\partial} \varphi$ is non-negative on $D$.
\end{proof}

\begin{proof}[Proof of Theorem \ref{thm:main-inverse-general}]
  Fix $z_0\in D$.  It follows from Lemma \ref{lem:affine} that both a translation of
  $z_0$ to the origin and an arbitrary invertible complex linear change of
  coordinates preserve Skoda divisibility.  More precisely, for
  $A\in\mathrm{GL}(n,\mathbb C)$ set
  $w=T(z)=A(z-z_0)$ and
  $\widetilde\varphi=\varphi\circ T^{-1}$.
  Proposition \ref{prop:trace} therefore implies
  $\operatorname{tr}i\partial\bar\partial\widetilde\varphi(0)\geqslant0$.

  Let
  $H_\varphi(z_0)=[\partial^2\varphi/\partial z_j\partial\bar z_k(z_0)]$.
  Since $z=z_0+A^{-1}w$, the chain rule gives
  \[ \left[ \frac{\partial^2 \tilde{\varphi}}{\partial w_j \partial \bar{w}_k}
     (0) \right]_{j, k} = (A^{- 1})^{\top} \left[ \frac{\partial^2
     \varphi}{\partial z_j \partial \bar z_k} (z_0) \right]_{j, k}
     \overline{A^{-1}} . \]
  Taking the trace and using the preceding non-negativity yields
  \[ \operatorname{tr} \left\{ (A^{- 1})^{\top} \left[ \frac{\partial^2
     \varphi}{\partial z_j \partial \bar z_k} (z_0) \right]_{j, k}
     \overline{A^{-1}} \right\} \geqslant 0, \]
  By cyclicity of the trace, this is equivalent to
  \[ \operatorname{tr} \left\{ \left[ \frac{\partial^2 \varphi}{\partial z_j \partial
     \bar z_k} (z_0) \right]_{j, k} \overline{A^{-1}}(A^{-1})^{\top} \right\} \geqslant
     0. \]
  As $A$ varies, the matrix
  $P=\overline{A^{-1}}(A^{-1})^\top$ ranges over all positive definite
  hermitian matrices.  Thus
  \[ \operatorname{tr} \left\{ \left[ \frac{\partial^2 \varphi}{\partial z_j \partial
     \bar z_k} (z_0) \right]_{j, k} P \right\} \geqslant 0 \]
  for every such $P$.  Given $v\in\mathbb C^n$ and $\delta>0$, choose
  $P=\delta\mathrm{Id}+vv^*$.  The trace inequality becomes
  \[ \langle i \partial\bar\partial \varphi (z_0) v, v \rangle + \delta
     \operatorname{tr} i \partial\bar\partial \varphi (z_0) \geqslant 0. \]
  Let $\delta\to 0$, we obtain
  \[ \langle i \partial\bar\partial \varphi (z_0) v, v \rangle \geqslant
     0 \]
  for every $v\in\mathbb C^n$.  Hence the Levi matrix of $\varphi$ is
  positive semi-definite at $z_0$.  Since $z_0$ was arbitrary,
  $\varphi$ is plurisubharmonic on $D$.
\end{proof}

\begin{remark}\label{rmk:partial-trace}
The slice limit also explains why the same proof cannot simply be run with
$r+1$ generators.  By Proposition \ref{rmk:r-generators}, the normalized
measure then converges along $L$, not at a point.  For
$F=|f|^2e^{-\varphi}$, the Levi contribution tends to the weighted average
\[
\frac{1}{M^4}
\frac{\int_{D\cap L}
\left(\sum_{j=1}^r\varphi_{j\bar j}(0,z'')\right)
|f(0,z'')|^2e^{-\varphi(0,z'')} dV_L(z'')}
{\int_{D\cap L}
|f(0,z'')|^2e^{-\varphi(0,z'')} dV_L(z'')},
\]
instead of the pointwise quantity
\[
\frac{1}{M^4}\sum_{j=1}^r\varphi_{j\bar j}(0).
\]
The limit contains values along $D\cap L$.  It therefore does
not imply pointwise non-negativity of the sum of $r$ Levi eigenvalues.  This
is the precise obstruction to replacing the $(n+1)$-generator test by an
$(r+1)$-generator test in the classical concentration proof.
\end{remark}
To summarize, we have the following result.
\begin{theorem}
\label{thm:equivalence}
Assume that $D\subset\mathbb C^n$ is bounded and pseudoconvex, that
$\varphi\in C^2(D)$, and that $A^2(D,\varphi)$ is non-trivial.  Then the
following conditions are equivalent:
\begin{enumeratenumeric}
 \item $\varphi$ is plurisubharmonic on $D$;
 \item every $f\in\mathcal O(D)$ is Skoda divisible
 with respect to $\varphi$;
 \item some non-zero $f\in A^2(D,\varphi)$ is Skoda divisible
 with respect to $\varphi$;
 \item every $f\in\mathcal O(D)$ is strongly Skoda divisible
 with respect to $\varphi$;
 \item some non-zero $f\in A^2(D,\varphi)$ is strongly Skoda divisible
 with respect to $\varphi$.
\end{enumeratenumeric}
\end{theorem}
\begin{proof}
The implication $(1)\Rightarrow(4)$ is the strong form of Skoda's division
theorem \cite{skoda1978}, see also \cite{li-meng-ning-zhou2025arxiv}.  Since $A^2(D,\varphi)\ne\{0\}$, condition (4)
implies (5), and Theorem \ref{thm:strong-converse} gives
$(5)\Rightarrow(1)$.  Thus (1), (4), and (5) are equivalent.  Skoda's Lemme
  Fondamental {\cite{skoda1972}}
 gives $(4)\Rightarrow(2)$ and
$(5)\Rightarrow(3)$, while $(2)\Rightarrow(3)$ is immediate.  Finally,
condition (3) contains the $(n+1)$-generator hypothesis of Theorem
\ref{thm:main-inverse-general}, hence $(3)\Rightarrow(1)$.
\end{proof}

\section{The converse of the surjectivity of trivial bundles}
\label{sec:higher-rank}

In this section we prove Theorem \ref{thm:flat-vector-converse}.  The method is essentially contained
in the rank-one case which is already proved. The proof of Theorem \ref{thm:flat-vector-converse} is reduced to the rank-one case by considering a sequence of special matrices.

We first record the pointwise orthogonal decomposition associated with a
surjective matrix
\[
 G:\underline{\mathbb C}^m\longrightarrow\underline{\mathbb C}^q.
\]
Its orthogonal right inverse is
\begin{equation}\label{equ:matrix-right-inverse}
 G^\dagger=G^\star(GG^\star)^{-1}.
\end{equation}
Thus $GG^\dagger=\operatorname{Id}_{\underline{\mathbb C}^q}$ and
\begin{equation*}
 |G^\dagger F|^2
 =\left\langle(GG^\star)^{-1}F,F\right\rangle.
\end{equation*}
Moreover, every lifting $H$ of $F$ has a unique orthogonal decomposition
\[
 H=G^\dagger F+u,\qquad u\in\ker G,
\]
and hence $|H|^2=|G^\dagger F|^2+|u|^2$.

\begin{proof}[Proof of Theorem \ref{thm:flat-vector-converse}]
For $q=1$, the first assertion is Theorem \ref{thm:main-inverse-general}.  We may
therefore assume that $q>1$.  Write
\[
 F=(f_1,\ldots,f_q).
\]
Without loss of generality, we may assume that
$f_1\not\equiv0$.  Since $F\in A^2(D,\underline{\mathbb C}^q,\varphi)$, we have
$f_1\in A^2(D,\varphi)$.

Fix $z_0\in D$, $A\in\operatorname{GL}(n,\mathbb C)$, and $M>0$, and
consider
\[
 g=(A(z-z_0),M)\in\mathcal O(D)^{\oplus(n+1)}.
\]
Its components have no common zero and $|g|^2\geqslant M^2$.  For $R>0$, define the
block matrix
\begin{equation}\label{equ:block-matrix}
 G_R=
 \begin{bmatrix}
  g&0\\
  0&R\operatorname{Id}_{\underline{\mathbb C}^{q-1}}
 \end{bmatrix}:
 \underline{\mathbb C}^{n+1}\oplus\underline{\mathbb C}^{q-1}\longrightarrow
 \underline{\mathbb C}\oplus\underline{\mathbb C}^{q-1}.
\end{equation}
It is pointwise surjective, and
\begin{equation*}
 G_RG_R^\star
 =\operatorname{diag}(|g|^2,R^2,\ldots,R^2),
 \qquad
 \det(G_RG_R^\star)=|g|^2R^{2(q-1)}.
\end{equation*}
Note that in this case $p$ in
\eqref{equ:matrix-basic-data} equals $n$.  Consequently,
\[
 \psi_{G_R,\varepsilon}
 =n(1+\varepsilon)\log|g|^2
 +2n(1+\varepsilon)(q-1)\log R.
\]
The last term is constant and cancels from the two sides of the division
estimate.

Let $g^*$ be the orthogonal right inverse in
\eqref{equ:orthogonal-right-inverse}.  From
\eqref{equ:matrix-right-inverse},
\begin{equation}\label{equ:block-canonical-lifting}
 G_R^\dagger F
 =\left(g^*f_1,\frac{f_2}{R},\ldots,\frac{f_q}{R}\right)
\end{equation}
and
\begin{equation}\label{equ:block-canonical-norm}
 |G_R^\dagger F|^2
 =\frac{|f_1|^2}{|g|^2}
 +\frac1{R^2}\sum_{\alpha=2}^q|f_\alpha|^2.
\end{equation}
All the relevant integrals are finite since $|g|\geqslant M$ is smooth on $\overline{D}$, and every component of $F$ belongs to the
weighted Bergman space.

The hypothesis gives $H_R\in\mathcal O(D)^{\oplus(n+q)}$ such that
$G_RH_R=F$ and
\[
 \int_D|H_R|^2e^{-\varphi-\psi_{G_R,\varepsilon}} dV
 \leqslant\left(1+\frac1\varepsilon\right)
 J_{\varphi,G_R,\varepsilon}(F).
\]
Put
\[
 u_R=H_R-G_R^\dagger F.
\]
By \eqref{equ:block-matrix},
\begin{equation*}
 \ker G_R=(\ker g)\oplus\{0\}^{q-1}.
\end{equation*}
We therefore identify $u_R$ with a section of
$S_g=\ker g\subset\underline{\mathbb C}^{n+1}$.  It then follows from
\eqref{equ:block-canonical-norm} that
\begin{align}
 &\int_D|u_R|^2
 e^{-\varphi-\psi_{g,\varepsilon}} dV \notag\\
 \leqslant&\frac1\varepsilon
 \int_D\left(\frac{|f_1|^2}{|g|^2}
 +\frac1{R^2}\sum_{\alpha=2}^q|f_\alpha|^2\right)
 e^{-\varphi-\psi_{g,\varepsilon}} dV,
 \label{equ:block-kernel-estimate}
\end{align}
where $\psi_{g,\varepsilon}=n(1+\varepsilon)\log|g|^2$.

Let $\beta_g$ denote the second fundamental form of $S_g$.  The last
$q-1$ components of \eqref{equ:block-canonical-lifting} are holomorphic,
whereas \eqref{equ:right-inverse-defect} gives
\[
 \bar\partial(g^*f_1)=-\beta_g^*f_1.
\]
Since $H_R$ is holomorphic, it follows that
\begin{equation}\label{equ:block-kernel-equation}
 \bar\partial_{S_g}u_R=\beta_g^*f_1.
\end{equation}

For fixed $g$ and $\varepsilon$, estimate
\eqref{equ:block-kernel-estimate} makes $(u_R)_{R\geqslant1}$ bounded in
the Hilbert space
\[
 L^2\bigl(D,S_g;e^{-\varphi-\psi_{g,\varepsilon}}\bigr).
\]
Along a sequence $R\to\infty$, it converges weakly to a section $u$.
Equation \eqref{equ:block-kernel-equation} passes to the limit in the sense
of distributions, and weak lower semi-continuity gives
\begin{equation*}
 \bar\partial_{S_g}u=\beta_g^*f_1,
 \qquad
 \int_D|u|^2e^{-\varphi-\psi_{g,\varepsilon}} dV
 \leqslant\frac1\varepsilon J_{\varphi,g,\varepsilon}(f_1).
\end{equation*}
By Proposition \ref{prop:dbar}, the function $f_1$ satisfies the scalar
classical Skoda estimate for every affine-radial row $g$ above and every
$\varepsilon>0$.  These are exactly the test rows used in the proof of
Theorem \ref{thm:main-inverse-general},
therefore $\varphi$ is plurisubharmonic on $D$.

It remains to prove the final equivalence.  If $D$ is pseudoconvex and
$\varphi$ is plurisubharmonic, Theorem \ref{thm:skoda-vector-bundle},
applied to the trivial bundles and with the canonical form $dz$ used to
trivialize $K_D$, gives Skoda divisibility for every $F$.  The
implication from (2) to (3) is immediate because the
weighted Bergman space is non-trivial, and (3) implies (1) by the first
part of the theorem.  This proves the equivalence.
\end{proof}

\section{Connection between the $L^2$ division and optimal $L^2$ extensions}\label{sec:optimal-extension}

In our previous work \cite{li-meng-ning-zhou2025arxiv}, we gave an alternative proof of Berndtsson’s complex Pr\'ekopa minimum principle using the equivalence between plurisubharmonicity and the strong $L^2$ division property. A similar argument shows that this principle also follows from Theorem \ref{thm:main-inverse-general}.

In this section, we show that the original Skoda $L^2$ division theorem yields a direct proof of an optimal $L^2$ extension theorem of Guan--Zhou \cite{guan-zhou2012, guan-zhou2014optimal},  
although not the full statement of latter. 
The readers may refer to \cite{zhou2015} for more details about optimal $L^2$ extensions. Our approach is stimulated by the work of Berndtsson--Lempert \cite{berndtsson-lempert2016} and Siu \cite{siu2011}. 
Throughout this section, we equip all trivial bundles with their standard Hermitian metrics.

\begin{theorem}\label{thm:optimal-extension}
Let $D\subset\mathbb C^n$ be a pseudoconvex domain,
$\varphi$ continuous, $1\leqslant r\leqslant n$.
Suppose that $s=(s_1,\ldots,s_r)\in\mathcal O(D)^{\oplus r}$ satisfies
\[
 |s|^2:=\sum_{j=1}^r|s_j|^2<1
\]
on $D$, and denote $S:=\{s_1=\cdots=s_r=0\}\neq\varnothing$. Suppose that
on $S$,
\[
d s_1\wedge\cdots\wedge ds_r\neq 0.
\]
Assume that, for every pseudoconvex domain $U\Subset D$ and every $k\geqslant 1$,
with $q=\binom{r+k-1}{k}$, every
$F\in\mathcal{O}(U)^{\oplus q}$ is Skoda divisible on $U$ with respect to $\varphi$.

Then every $f\in\mathcal O(S)$ satisfying
\[
 \int_S\frac{|f|^2e^{-\varphi}}
 {|ds_1\wedge\cdots\wedge ds_r|^2}dV_S<\infty,
\]
there exists $\widetilde f\in\mathcal O(D)$ such that
\begin{equation}\label{equ:ci-extension}
 \widetilde f|_S=f,\qquad
 \int_D|\widetilde f|^2e^{-\varphi}dV
 \leqslant\frac{\pi^r}{r!}
 \int_S\frac{|f|^2e^{-\varphi}}
 {|ds_1\wedge\cdots\wedge ds_r|^2}dV_S,
\end{equation}
where $dV$ and $d V_S$ are the Lebesgue measure on $D$ and $S$ respectively.
\end{theorem}

For an integer $k\geqslant1$ and $c>0$, set
\[
 \Phi_c=\varphi+c|s|^{2k},\qquad
 \llangle u,v\rrangle_c=\int_Uu\overline v e^{-\Phi_c}dV,
 \qquad \|u\|_c^2=\llangle u,u\rrangle_c.
\]
Since $|s|<1$,
the norms $\|\cdot\|_c$ and $\|\cdot\|_\varphi$ are equivalent,
and the corresponding weighted Bergman spaces coincide.

The key step to use Skoda divisibility is the following approximation result.
\begin{proposition}\label{prop:ci-approximation}
For every integer $k\geqslant1$, every $c>0$, and every
$h\in A^2(U,\varphi)$, there exists $v\in A^2(U,\varphi)$ such that
\begin{equation}\label{equ:ci-approximation}
 \big\||s|^{2k}h-v\big\|_c^2
 \leqslant\frac1c\int_U|s|^{2k}|h|^2e^{-\Phi_c}dV.
\end{equation}
\end{proposition}

\begin{proof}
For a multi-index
$\alpha=(\alpha_1,\ldots,\alpha_r)\in\mathbb{Z}_{\geqslant0}^r$, write
$\alpha!=\alpha_1!\cdots\alpha_r!$,
$s^\alpha=s_1^{\alpha_1}\cdots s_r^{\alpha_r}$ and
$q=\binom{r+k-1}{k}$. For $\delta>0$ and
$h\in A^2(U,\varphi)$, set
\begin{equation*}
 G_\delta=
 \left[\operatorname{Id}_{\underline{\mathbb C}^{q}},
       \ \delta\left(\sqrt{\frac{k!}{\alpha!}}s^\alpha
                       \right)_{|\alpha|=k}\right],
 \qquad
 F=\left(\sqrt{\frac{k!}{\alpha!}}s^\alpha h
                       \right)_{|\alpha|=k},
\end{equation*}
where the last column of $G_{\delta}$ is indexed in the same order as the components of
$F$. It follows from
\[
 \sum_{|\alpha|=k}\frac{k!}{\alpha!}|s^\alpha|^2=|s|^{2k},
\]
that $G_\delta$ is pointwise surjective,
\begin{equation*}
 \det(G_\delta G_\delta^\star)=1+\delta^2|s|^{2k},\qquad
 (G_\delta G_\delta^\star)^{-1}F
       =\frac{F}{1+\delta^2|s|^{2k}},
\end{equation*}
and
\[
 J_{\varphi,G_\delta,\varepsilon}(F)
 =\int_U|s|^{2k}|h|^2
       (1+\delta^2|s|^{2k})^{-2-\varepsilon}e^{-\varphi}dV.
\]
 For $\delta>0$, it follows from Skoda divisibility that, with
 $\varepsilon=c/\delta^2$, there exists $H=(h_1,\ldots,h_{q+1})\in\mathcal{O}
 (U)^{\oplus (q+1)}$ such that $G_\delta H=F$ and
 \begin{equation}\label{equ:ext-division}
         \int_U |H|^2(1+\delta^2 |s|^{2k})^{-1-\varepsilon}e^{-\varphi}d V
    \leqslant (1+\frac{1}{\varepsilon})J_{\varphi,G_{\delta},\varepsilon}(F).
 \end{equation}

The equation
$G_\delta H=F$ identifies the first $q$ components of $H$ with
\[
 \left(\sqrt{\frac{k!}{\alpha!}}s^\alpha
              (h-\delta h_{q+1})\right)_{|\alpha|=k}.
\]
Therefore, one has
\begin{align*}
 |H|^2
 &=|s|^{2k}|h-\delta h_{q+1}|^2+|h_{q+1}|^2\\
 &=\frac{|s|^{2k}|h|^2}{1+\delta^2|s|^{2k}}
   +\delta^2(1+\delta^2|s|^{2k})
      \left|\frac{h_{q+1}}\delta
             -\frac{|s|^{2k}h}{1+\delta^2|s|^{2k}}\right|^2.
\end{align*}
It then follows from \eqref{equ:ext-division} that
\begin{align}
 &\int_U\left|\frac{h_{q+1}}\delta
              -\frac{|s|^{2k}h}{1+\delta^2|s|^{2k}}\right|^2
        (1+\delta^2|s|^{2k})^{-\varepsilon}e^{-\varphi}dV\notag\\
 &\qquad\leqslant\frac1c\int_U|s|^{2k}|h|^2
        (1+\delta^2|s|^{2k})^{-2-\varepsilon}e^{-\varphi}dV.
 \label{equ:ci-division-limit}
\end{align}
 Since $\log(1+t)\leqslant t$ for $t\geqslant0$,
\[
 (1+\delta^2|s|^{2k})^{-\varepsilon}=(1+\delta^2|s|^{2k})^{-c/\delta^2}
 \geqslant e^{-c|s|^{2k}},
\]
and  $\lim_{\delta\to 0}(1+\delta^2|s|^{2k})^{-\varepsilon}
       = e^{-c|s|^{2k}},$
it follows from the dominated convergence theorem and \eqref{equ:ci-division-limit} that $h_{q+1}/\delta$ is bounded in
$A^2(U,\Phi_c)$. Therefore we can find a subsequence of $\{h_{q+1}/\delta\}$ converging to $v\in A^2(U,\varphi)$ satisfying \eqref{equ:ci-approximation}.
\end{proof}
With the above approximation result, we have the following property, which is a counterpart of the variation of domains used in \cite{berndtsson-lempert2016}.
\begin{proposition}\label{prop:ci-minimum-concavity}
Let $\ell$ be a non-zero finite
linear combination of point evaluations at points of $S\cap U$, i.e., for some $N\geqslant1$,
\[
\ell=\sum_{j=1}^{N}\lambda_j \operatorname{EV}_{z_j}, \quad z_j\in S\cap U.
\]
We regard $\ell$ as a continuous linear functional on $A^2(U,\varphi)$, put
\begin{equation*}
 M(c)=\min\left\{\|h\|_c^2:
                 h\in A^2(U,\varphi),\ \ell(h)=1\right\},
\end{equation*}
and denote the unique minimizer by $f_c$.
Then
\begin{enumerate}
    \item The maps $c\mapsto f_c$ and $c\mapsto M(c)$ are smooth on $(0,\infty)$.
    \item The function
\[
 x\longmapsto\log M(e^{-x})
\]
is concave on $\mathbb R$.
\end{enumerate}
\end{proposition}

\begin{proof}
  1. As we have mentioned that $A^2(U,\varphi)$ coincides with $A^2(U,\Phi_c)$ in
    the sense of sets, we denote by $A^2$ for both $A^2(U,\varphi)$ and  $A^2(U,\Phi_c)$.
    Let $\| \cdot \|_0 := \| \cdot \|_{\varphi}$ be a fixed Hilbert norm on $A^2$, define the bounded operator $T_c$ on $A^2$ by
    \[
    \llangle T_c F,G \rrangle_0=\llangle F,G \rrangle_c, \quad F,G\in A^2.
    \]
    It follows that $e^{-c}\operatorname{Id}\leqslant  T_c\leqslant \operatorname{Id}$. We denote $\|\cdot\|_{\operatorname{op}}$ as the operator norm. Let $P_\varphi:L^2(U,e^{-\varphi})\to A^2$ be the orthogonal projection. Then
    $T_cF=P_\varphi(e^{-c|s|^{2k}}F)$. Define
    $T_c'F=-P_\varphi(|s|^{2k}e^{-c|s|^{2k}}F)$.
    Since $|s|\leqslant1$ and $\|P_\varphi\|\leqslant1$, $T_c'$ is a bounded operator on $A^2$.
    Since for $\lambda\ne0$ with $c+\lambda>0$,
    \[
        \left\| \frac{T_{c+\lambda} - T_c}{\lambda} - T'_c \right\|_{\operatorname{op}} \leqslant \sup_{\|F\|_0=1} \left\| \left( \frac{e^{-(c+\lambda)|s|^{2k}} - e^{-c|s|^{2k}}}{\lambda} + |s|^{2k} e^{-c|s|^{2k}} \right) F \right\|_0,
    \]
    and
    \[
    \left| \frac{e^{-(c+\lambda)|s|^{2k}} - e^{-c|s|^{2k}}}{\lambda} + |s|^{2k} e^{-c|s|^{2k}} \right| \leqslant |\lambda|,
    \]
    we have
    \[
    \lim_{\lambda\to 0}\left\| \frac{T_{c+\lambda} - T_c}{\lambda} - T'_c \right\|_{\operatorname{op}}=0,
    \]
    that is $T_c$ is differentiable. Therefore, $T_c$ is smooth by induction.

    Let $r_0$ be the Riesz representative of $\ell$ for $\llangle\cdot,\cdot\rrangle_0$, its representative for $\llangle\cdot,\cdot\rrangle_c$ is $r_c=T^{-1}_c r_0$, hence
    \[
    f_c=\frac{r_c}{\ell (r_c)},\quad \ell(r_c)=\|r_c\|_c^2,
    \]
    it follows that $c\mapsto f_c$ and $M(c)$ are smooth on $(0,\infty)$.

  2. We write  $f'_c$ as $d f_c/d c$. Since $\ell (f_c)=1$, we have $f'_c\in\ker\ell$ and
    \begin{equation}\label{equ:ci-minimizer-derivative}
    \llangle f'_c,h\rrangle_c=\llangle |s|^{2k}f_c,h\rrangle_c,
    \end{equation}
    where $h\in \ker \ell$.
    In particular, $\llangle |s|^{2k}f_c,f'_c\rrangle_c=\|f'_c\|^2_c$. Therefore, it follows from differentiating $M(c)=\|f_c\|^2_c$ that
    \begin{equation}\label{equ:ci-minimum-derivatives}
    \begin{aligned}
        M'(c)&=-\int_U|s|^{2k}|f_c|^2e^{-\Phi_c}dV,\\
        M''(c)&=\big\||s|^{2k}f_c\big\|_c^2
                         -2\|f_c'\|_c^2.
 \end{aligned}
\end{equation}
Set
\[
 u_c=|s|^{2k}f_c-f_c'+\frac{M'(c)}{M(c)}f_c.
\]
By \eqref{equ:ci-minimizer-derivative}, $u_c$ is orthogonal to
$\ker\ell$. The first identity in
\eqref{equ:ci-minimum-derivatives} shows that it is also orthogonal
to $f_c$, hence to all of $A^2$. That is, the decomposition
\[
|s|^{2k}f_c=f'_c-\frac{M'(c)}{M(c)}f_c+u_c
\]
is orthogonal and
\[
\| |s|^{2k}f_c \|_c^2=\| f'_c\|_c^2+\frac{M'(c)^2}{M(c)}+\| u_c \|_c^2.
\]
Consequently,
\[
 M''(c)-\frac{M'(c)^2}{M(c)}
       =\|u_c\|_c^2-\|f_c'\|_c^2.
\]
Since for any $\tilde{h}\in A^2$, we have $\| |s|^{2k}f_c-\tilde{h} \|_c^2\geqslant \|u_c\|_c^2$.
Apply
Proposition \ref{prop:ci-approximation} to $f_c$, we have
\[
\| u_c \|_c^2\leqslant -\frac{M'(c)}{c}.
\]
For $c=e^{-x}$, it follows from the above estimate that
\begin{align*}
 \frac{d^2}{dx^2}\log M(e^{-x})
 &=\frac{cM'(c)}{M(c)}
   +\frac{c^2}{M(c)}
          \bigl(\|u_c\|_c^2-\|f_c'\|_c^2\bigr)\\
 &\leqslant-\frac{c^2}{M(c)}\|f_c'\|_c^2\leqslant0.
\end{align*}
\end{proof}

\begin{proposition}\label{prop:ci-concentration}
Suppose in addition that $U$ is smoothly bounded and that
$\partial U$ meets $S$ transversely. Let $f_0$ be holomorphic on a
neighborhood of $\overline U$ with $f_0|_{S\cap U}=f$. Then
\begin{equation}\label{equ:ci-concentration}
 \lim_{c\to\infty}c^{r/k}\|f_0\|_c^2
 =\frac{\pi^r}{r!}\Gamma(1+\frac{r}{k})
   \int_{S\cap U}\frac{|f|^2e^{-\varphi}}
      {|ds_1\wedge\cdots\wedge ds_r|^2}dV_S.
\end{equation}
Moreover, for every non-zero finite $\ell=\sum_{j=1}^{N}\lambda_j \operatorname{EV}_{z_j}$ with $z_1,\cdots z_N\in S\cap U$ and $M(c)$ defined above, there
exists $C>0$ independent of $c$ such that, for all sufficiently
large $c$ and all $h\in A^2(U,\varphi)$,
\begin{equation*}
 |\ell(h)|^2\leqslant Cc^{r/k}\|h\|_c^2,
 \qquad M(c)\geqslant C^{-1}c^{-r/k}.
\end{equation*}
\end{proposition}
\begin{proof}
  1. Choose $\delta>0$ small enough such that
    \[
    V_\delta=\{z\in U : |s(z)|<\delta\}
    \]
    is a neighborhood of $S\cap U$ in $U$ and $s$ is a submersion on a neighborhood of $\overline{V_\delta}$.
    Cover $S\cap\overline U$ by finitely many coordinate neighborhoods with coordinates
    $(y,w)\in\mathbb C^{n-r}\times\mathbb C^r$, where $w=s(z)$.
    After shrinking $\delta$, choose a smooth partition of unity $(\rho_a)$
    on a neighborhood of $\overline{V_\delta}$ subordinate to these neighborhoods,
    with each $\rho_a$ compactly supported in its coordinate neighborhood.
    Denote the inverse coordinate maps by $z_a(y,w)$. Set
    \[
    J(w)=\int_{s^{-1}(w)\cap U}|f_0|^2e^{-\varphi}\frac{dV_{s^{-1}(w)}}{|ds_1\wedge\cdots\wedge ds_r|^2},
    \]
    and therefore
    \[
    \int_{U\cap V_\delta}|f_0|^2e^{-\Phi_c}\,dV=\int_{|w|<\delta}e^{-c|w|^{2k}}J(w)\,dV_w.
    \]
    Since $\partial U$ intersects $S$ transversely, in each coordinate neighborhood
    the set $\{y:z_a(y,0)\in\partial U\}$ has measure zero. Therefore
    $\mathbb I_U(z_a(y,w))\to\mathbb I_U(z_a(y,0))$ for almost every $y$ as $w\to0$.
    The coarea formula in these coordinates gives
    \begin{equation}\label{equ:slice-integral}
    \begin{aligned}
    J(w)&=\sum_a\int_{\mathbb C^{n-r}}\mathbb I_U(z_a(y,w))
    \rho_a(z_a(y,w))|f_0(z_a(y,w))|^2e^{-\varphi(z_a(y,w))}\\
    &\qquad\qquad\times
    \left|\det\frac{\partial z_a}{\partial(y,w)}\right|^2dV_y
    \longrightarrow J(0),
    \end{aligned}
    \end{equation}
    as $w\to0$, by the dominated convergence theorem applied in each chart.
    Here each integrand is extended by zero outside its coordinate neighborhood;
    the partition of unity gives fixed compact supports and uniform bounds.

    Now set $w=c^{-1/(2k)}\xi$ and thus
    \[
    \begin{split}
         &\lim_{c\to\infty}c^{r/k}\int_{U\cap V_\delta}|f_0|^2e^{-\Phi_c}\,dV\\
         =&\lim_{c\to\infty}\int_{|\zeta|<\delta c^{1/(2k)}}e^{-|\zeta|^{2k}}J(c^{-1/(2k)}\zeta)dV_\zeta\\
         =&J(0)\int_{\mathbb{C}^r}e^{-|\xi|^{2k}}d V_\xi\\
         =&J(0) \frac{2\pi^r}{(r-1)!}\int_0^\infty e^{-t^{2k}}t^{2r-1}\,dt\\
         =&\frac{\pi^r}{r!}\Gamma(1+r/k)\int_{S\cap U}\frac{|f|^2e^{-\varphi}}{|ds_1\wedge\cdots\wedge ds_r|^2}\,dV_S.
    \end{split}
    \]
    It then follows from \eqref{equ:slice-integral} that
    \[
    c^{r/k}\int_{U\cap V_\delta}|f_0|^2e^{-\Phi_c}\,dV\to J(0)\int_{\mathbb C^r}e^{-|\zeta|^{2k}}\,dV_\zeta.
    \]
    Since
    \[
    \|f_0\|_c^2=\int_{U\cap V_\delta}|f_0|^2e^{-\Phi_c}\,dV+\int_{U\setminus V_\delta}|f_0|^2e^{-\Phi_c}\,dV
    \]
    and on $U\setminus V_\delta$,
    \[
    c^{r/k}\int_{U\setminus V_\delta}|f_0|^2e^{-\Phi_c}\,dV\le c^{r/k}e^{-c\delta^{2k}}\int_U |f_0|^2e^{-\varphi}\,dV\to 0,
    \]
    as $c\to\infty$, we have \eqref{equ:ci-concentration} proved.

  2. To prove the evaluation estimate, fix $z_0\in S\cap U$. Choose a local coordinate neighborhood $(y,w)\in \Delta_R^{n-r}\times\Delta_\rho^r$ with $w=s(z)$ centered at $z_0$, where $\Delta_R=\{\xi\in\mathbb{C}:|\xi|<R\}$. For sufficiently large $c$, put
\[
 \rho_c=r^{-1/2}c^{-1/(2k)}\leqslant\rho.
\]
On $\Delta_R^{\,n-r}\times\Delta_{\rho_c}^{\,r}$ we have
\[
 |w|^2=\sum_{j=1}^r|w_j|^2
       \leqslant r\rho_c^2=c^{-1/k},
 \quad c|w|^{2k}\leqslant1,
 \quad e^{-c|w|^{2k}}\geqslant e^{-1}.
\]
For any $h\in A^2(U,\varphi)$, we have
\[
 \begin{aligned}
 |h(z_0)|^2
 &\leqslant
  \frac{1}{\pi^{n-r}R^{2(n-r)}\,\pi^r\rho_c^{2r}}
  \int_{\Delta_R^{\,n-r}\times\Delta_{\rho_c}^{\,r}}
                     |h(y,w)|^2dV_y\,dV_w\\
 &=\frac{r^r}{\pi^nR^{2(n-r)}}c^{r/k}
  \int_{\Delta_R^{\,n-r}\times\Delta_{\rho_c}^{\,r}}
                     |h(y,w)|^2dV_y\,dV_w,
 \end{aligned}
\]
Let $\det\frac{\partial(y,w)}{\partial z}$ denote the Jacobi determinant of the local coordinate $(y,w)$, we have
\[
\begin{split}
    |h(z_0)|^2\leqslant&\frac{r^r}{\pi^nR^{2(n-r)}}c^{r/k}
  \int_{\{z:(y(z),w(z))\in\Delta_R^{\,n-r}\times\Delta_{\rho_c}^{\,r}\}}
                     |h(z)|^2|\det\frac{\partial(y,w)}{\partial z}|^2d V_z\\
       \leqslant& \frac{r^r}{\pi^nR^{2(n-r)}}c^{r/k}\cdot\frac{e}{b}\|h\|^2_c:=C_{z_0}c^{r/k}\|h\|^2_c.
\end{split}
\]
where $b:=\min\{
       e^{-\varphi(z(y,w))}|\det\frac{\partial z}{\partial(y,w)}|^2:(y,w)\in\overline{\Delta_R^{\,n-r}\times\Delta_\rho^{\,r}}\}>0$.

For $\ell(h)=\sum_{j=1}^{N}\lambda_j h(z_j)$ with $N\geqslant1$ and $z_j\in S\cap U$, it follows that
\[
 \begin{aligned}
 |\ell(h)|^2
 &\leqslant\left(\sum_{j=1}^N|\lambda_j|\sqrt{C_{z_j}}\right)
       \left(\sum_{j=1}^N
                  \frac{|\lambda_j|}{\sqrt{C_{z_j}}}|h(z_j)|^2\right)\\
 &\leqslant\left(\sum_{j=1}^N|\lambda_j|\sqrt{C_{z_j}}\right)^2
                          c^{r/k}\|h\|_c^2,
 \end{aligned}
\]
and $M(c)\geqslant\left(\sum_{j=1}^N|\lambda_j|\sqrt{C_{z_j}}\right)^{-2}c^{-r/k}$ follows by choosing $h$ such that $\ell(h)=1$.
\end{proof}

\begin{proof}[Proof of Theorem \ref{thm:optimal-extension}]

Since $D$ is pseudoconvex, there exists $f_0\in\mathcal O(D)$ such that
$f_0|_S=f$. Choose an increasing
sequence of relatively compact pseudoconvex domains
$U_\nu$ with smooth boundaries exhausting $D$ , such that $\partial U_\nu$ meets $S$
transversely. Fix $U=U_\nu$ and $k\geqslant1$. Since $f_0$ and $\varphi$
are bounded on $\overline U$, we have $f_0\in A^2(U,\varphi)$.

Let $\ell$ be a non-zero finite linear combination of point
evaluations on $S\cap U$, and let $M(c)$ be its constrained
minimum. Proposition \ref{prop:ci-minimum-concavity} shows that
\[
 x\longmapsto\log M(e^{-x})-\frac rk x
\]
is concave, and Proposition \ref{prop:ci-concentration} bounds
it below as
$x\to-\infty$. It is therefore nonincreasing and thus for every $c_0>0$,
\begin{equation}\label{equ:ci-monotonicity}
 c_0^{r/k}M(c_0)
       \leqslant\liminf_{c\to\infty}c^{r/k}M(c).
\end{equation}
If $\ell(f_0)\ne0$, the function $f_0/\ell(f_0)$ is an admissible
competitor for $M(c)$. Combining \eqref{equ:ci-monotonicity} and
\eqref{equ:ci-concentration}, we obtain
\begin{equation}\label{equ:ci-functional-bound}
 |\ell(f_0)|^2M(c_0)
 \leqslant\frac{\pi^r}{r!}\Gamma(1+\frac{r}{k})c_0^{-r/k}
   \int_{S\cap U}\frac{|f|^2e^{-\varphi}}
      {|ds_1\wedge\cdots\wedge ds_r|^2}dV_S,
\end{equation}
the same inequality is immediate when $\ell(f_0)=0$.

Now fix $c=c_0$, and
$\Phi_{c_0}=\varphi+c_0|s|^{2k}$. It follows that,
\[
 M(c_0)^{-1}
   =\sup\{|\ell(h)|^2:h\in A^2(U,\varphi), \|h\|_{c_0}\leqslant 1\}=\|\ell\|_{c_0}^2.
\]
The set $I$ containing all holomorphic functions in $A^2(U,\varphi)$ vanishing on $S\cap U$ forms a
closed subspace. The functionals of the form $\ell=\sum_{j=1}^N \lambda_j \operatorname{EV}_{z_j}$ with $N\geqslant1$ and $z_j\in S\cap U$ are dense in $I^*=\{\ell\in A^2(U,\varphi)^*:\ell|_I=0\}$. It then follows from \eqref{equ:ci-functional-bound} and the density of finite combinations of evaluations on $S\cap U$ that
\begin{equation}\label{equ:ci-weighted-extension}
 \begin{split}
 \min_{\substack{h\in A^2(U,\varphi)\\h|_{S\cap U}=f}}
    \|h\|_{c_0}^2&=\sup_{0\ne\ell\in I^*}\frac{|\ell(f_0)|^2}{\|\ell\|^2_{c_0}}\\
    &\leqslant\frac{\pi^r}{r!}\Gamma(1+\frac{r}{k})c_0^{-r/k}
   \int_{S\cap U}\frac{|f|^2e^{-\varphi}}
      {|ds_1\wedge\cdots\wedge ds_r|^2}dV_S.
 \end{split}
\end{equation}

Since $|s|<1$, we have
$\|h\|_\varphi^2\leqslant e^{c_0}\|h\|_{c_0}^2$. The function
$e^{c_0}c_0^{-r/k}$ is minimized at $c_0=r/k$. Hence it follows from \eqref{equ:ci-weighted-extension} that there exists an extension
$\widetilde f_{\nu,k}$ on $U_\nu$ satisfying
\begin{equation*}
 \int_{U_\nu}|\widetilde f_{\nu,k}|^2e^{-\varphi}dV
 \leqslant\frac{\pi^r}{r!}\Gamma(1+\frac{r}{k})
              \left(\frac{ek}{r}\right)^{r/k}
 \int_{S\cap U_\nu}
      \frac{|f|^2e^{-\varphi}}
      {|ds_1\wedge\cdots\wedge ds_r|^2}dV_S.
\end{equation*}
For fixed $k$, the right-hand side is bounded independently of
$\nu$. Therefore, there exists a subsequence of $\{\widetilde f_{\nu,k}\}_{\nu}$, which is still denoted by itself,
converging locally uniformly on $D$ to a holomorphic
extension $\widetilde f_k$ of $f$. It follows from Fatou's lemma that
\[
 \int_D|\widetilde f_k|^2e^{-\varphi}dV
 \leqslant\frac{\pi^r}{r!}\Gamma(1+\frac{r}{k})
             \left(\frac{ek}{r}\right)^{r/k}
 \int_S
     \frac{|f|^2e^{-\varphi}}
     {|ds_1\wedge\cdots\wedge ds_r|^2}dV_S.
\]
Since
\[
 \log\Gamma(1+r/k)+\frac rk\log\frac{ek}{r}
       \longrightarrow0\qquad\text{as }k\to\infty.
\]
We could find a subsequence of $\{\widetilde f_k\}$, which is still denoted by itself, converging locally uniformly on $D$ to a holomorphic
extension $\widetilde f$ of $f$.
It again follows by Fatou's lemma that
$\widetilde f$ satisfying \eqref{equ:ci-extension}.
\end{proof}
\begin{remark}
    Recently Albesiano \cite{albesiano2025} and Takakura \cite{takakura2025} gave proofs of the non-optimal and optimal $L^2$ extension theorems, respectively, both via $L^2$ division. However, the estimates used in their proofs are different from that of Skoda's original $L^2$ estimates.
\end{remark}
\begin{remark}
    This shows that if every tuple of holomorphic functions is Skoda divisible with respect to a continuous function $\varphi$, then the optimal $L^2$ extension theorem holds with respect to $\varphi$. Consequently, by \cite{deng-ning-wang-zhou2023, deng-wang-zhang-zhou2024new-char}, $\varphi$ is plurisubharmonic. Thus, the $C^2$-regularity assumption in Theorem \ref{thm:flat-vector-converse} can be weakened to continuity.
\end{remark}

\end{sloppypar}
\end{document}